\documentclass[a4paper]{amsart}

\usepackage{amsmath,amsthm}
\usepackage{amsfonts}
\usepackage{graphicx}
\usepackage{todonotes}
\usepackage{enumitem}
\usepackage{subfigure}
\usepackage[a4paper, total={7in, 9in}]{geometry}

\DeclareMathOperator{\cof}{cof}
\DeclareMathOperator{\Div}{div}

\newtheorem{theorem}{Theorem}
\newtheorem{lemma}{Lemma}

\newcommand{\dv}{\operatorname{div}}
\newcommand{\velo}{{\mathbf v}}
\newcommand{\welo}{{\mathbf w}}
\newcommand{\wass}{{\mathbf{W}_2}}
\newcommand{\mf}{{\mathcal M}}
\newcommand{\nml}{{\boldsymbol \nu}}
\newcommand{\energy}{\mathcal{E}}
\newcommand{\nrg}{\mathcal{E}}
\newcommand{\ent}{\mathcal{H}}
\newcommand{\setR}{{\mathbb R}}
\newcommand{\prb}{{\mathcal P}}
\newcommand{\dd}{\,\mathrm{d}}
\newcommand{\dn}{\mathrm{d}}
\newcommand{\dt}{{\Delta t}}
\newcommand{\dff}{\operatorname D}
\newcommand{\id}{\operatorname{id}}
\newcommand{\eins}{{\mathbf 1}}
\newcommand{\grd}{\operatorname{grad}}
\newcommand{\ansatz}{{\mathcal X_{\mathbf\xi}}}
\newcommand{\vansatz}{{\mathcal V}}
\newcommand{\bfx}{{\mathbf x}}
\newcommand{\bfxi}{{\mathbf\xi}}
\newcommand{\xtoX}{\mathbf X_{\mathbf\xi}}
\newcommand{\xtorho}{\mathbf \rho_{\mathbf\xi}}
\newcommand{\cadlag}{c\`{a}dl\`{a}g}
\newcommand{\indy}{\mathbf 1}
\newcommand{\aint}{\mathop{\,\rlap{-}\!\!\!\int}\nolimits}
\newcommand{\tr}{\operatorname{tr}}
\newcommand{\R}{\setR}

\newcommand{\eps}{\epsilon}
\newcommand{\C}{\mathcal{C}}
\newcommand{\rd}{\mathrm{d}}
\newcommand{\Dtime}{1}

\newcommand{\Rd}{{\mathbb{R}^d}}
\newcommand{\grad}{\nabla}
\def\P{{\mathcal P}}

\title{Lagrangian schemes for Wasserstein Gradient Flows}
\author{Jose A. Carrillo}
\address[J. A. Carrillo]{Department of Mathematics, Imperial College London, Huxley Building, SW72AZ London}
\email{carrillo@imperical.ac.uk}
\author{Daniel Matthes}
\address[D. Matthes]{Department of Mathematics, TU Munich, Boltzmannstr. 3, 85748 Garching}
\email{matthes@ma.tum.de}
\author{Marie-Therese Wolfram}
\address[M.-T. Wolfram]{Mathematics Institute, University of Warwick, CV47AL Coventry, and Radon Institute of Computational and Applied Mathematics, Altenbergerstr. 69, 4040 Linz, Austria}
\email{m.wolfram@warwick.ac.uk}

\title{Lagrangian schemes for Wasserstein Gradient Flows}

\begin{document}

\begin{abstract}  
  This paper reviews different numerical methods for specific examples of Wasserstein gradient flows:
  we focus on nonlinear Fokker-Planck equations,
  but also discuss discretizations of the parabolic-elliptic Keller-Segel model
  and of the fourth order thin film equation.
  The methods under review are of Lagrangian nature, that is,
  the numerical approximations trace the characteristics of the underlying transport equation
  rather than solving the evolution equation for the mass density directly.
  The two main approaches are based on integrating the equation for the Lagrangian maps on the one hand,
  and on solution of coupled ODEs for individual mass particles on the other hand.
\end{abstract}

\maketitle


\section{Introduction}\label{sec:intro}
\noindent
In most general terms, $L^2$-Wasserstein gradient flows are evolution equations
for a time-dependent probability density $\rho_{(\cdot)}:[0,T]\times\Omega\to\setR_{\ge0}$
on a domain $\Omega \subset \mathbb{R}^d$ that can be written as follows:
\begin{align}
  \label{gf:0}
  \partial_t\rho_t + \dv(\rho_t\,\velo_t) = 0, \quad \text{with} \quad \velo_t = -\left[\nabla\frac{\delta\nrg}{\delta\rho}(\rho_t)\right],
\end{align}
subject to no-flux boundary conditions. In equation \eqref{gf:0}  the probability density $\rho$ is transported along a time-dependent gradient vector field $\velo_t:\Omega\to\setR^d$;
and the pressure generating that vector field depends on $\rho$ through the variational derivative of a free energy functional $\nrg$.
In this chapter, we shall review Lagrangian methods for computing the solutions of \eqref{gf:0}.

\subsection{Examples for \eqref{gf:0}}
The most prominent example of the form \eqref{gf:0} is the linear heat equation $\partial_t\rho_t=\Delta\rho_t$.
Indeed, choosing $\nrg$ as Boltzmann's entropy functional, $\nrg(\rho)=\ent_1(\rho):=\int_\Omega\rho\log\rho\dd x$,
one obtains:
\begin{align*}
  \frac{\delta\ent_1}{\delta\rho}(\rho_t) = \log\rho_t,
  \quad\text{and thus}\quad
  \partial_t\rho_t=\dv(\rho_t\nabla\log\rho_t) = \Delta\rho_t.
\end{align*}
Note that this is a hydrodynamical view on the heat equation,
which is complementary to the more classical understanding in probabilistic terms
as stochastic motion of Brownian particles.

Among the many further evolution equations of the form \eqref{gf:0}, 
there are non-linear and non-local Fokker-Planck equations,
\begin{align}
  \label{gf:FP}
  \partial_t\rho_t = \Delta \Phi(\rho_t) + \dv\big(\rho_t\,[\nabla V + \rho_t\ast \nabla W]\big),
\end{align}
where $\Phi:\setR_{\ge0}\to\setR_{\ge0}$ is a nonlinearity subject to certain conditions
(e.g., $\Phi(\rho)=\rho^m$ with arbitary $m>0$ is allowed),
$V\in C^2(\Omega)$ is an external potential,
and $W\in C^2(\setR^d)$ represents the potential of a non-local inter-particle interaction.
The corresponding entropy functional is
\begin{align}
  \label{gf:FPenergy}
  \nrg(\rho) = \ent_{h,V,W}(\rho) := \int_\Omega \big[h(\rho) + \rho V + \frac12 \rho( W\ast\rho)\big]\dd x,
\end{align}
with an entropy density $h:\setR_{\ge0}\to\setR$ such that $\rho h''(\rho)=\Phi'(\rho)$.
In the variational context, equation \eqref{gf:FP} is augmented with the natural boundary conditions,
which are no-flux,
\begin{align}\label{e:bc}
 \rho \nabla \big[h'(\rho) + V + W*\rho\big]\cdot \nml = 0 \quad \text{ for all } x \in \partial \Omega,\,t>0,
\end{align}
with $\nml$ the unit outwards normal to the boundary of $\Omega$.
For an overview on the vast field of applications of porous medium or filtration equations,
which are \eqref{gf:FP} with $V\equiv0$ and $W\equiv0$,
we refer to the book of Vazquez \cite{book:Vazquez}.
More recently, a rich theory has been developed for \eqref{gf:FP} with irregular potentials $W$.
A case of particular interest is the parabolic-elliptic Keller-Segel model,
which is \eqref{gf:FP} with linear diffusion $\Phi(r)=r$, with $V\equiv0$,
and $W$ given by the Newtonian potential, that is $W(z)=\frac1{2\pi}\log|z|$ in dimension $d=2$,
see \cite{DP,BDP,gf-BCC,gf-BCC2,CCY} and the references therein.
Complementary to that, a solution theory has been developed, see e.g. \cite{gf-CFetc},
for quite general interaction potentials $W$, even in the absense of diffusion, $\Phi\equiv0$.
Prominent applications are in studies of the collective dynamics of bird flocks or fish schools,
see for example \cite{BCL, HB2010, gf-CFetc, BGL2012,BCLR2,BCLR1},
but such equations also find applications in physics,
for example for granular media \cite{T2000,CMV2003,LT2004,cmcv-06}
or in material sciences e.g \cite{HP2005}.

Examples of fourth order equations of type \eqref{gf:0} include
\begin{align}
  \label{gf:TF}
  \partial_t\rho_t + \dv(\rho_t\nabla\Delta\rho_t)=0,
  \quad\text{and}\quad
  \partial_t\rho_t + \dv\left(\rho_t\nabla\frac{\Delta\sqrt{\rho_t}}{\sqrt{\rho_t}}\right)=0.
\end{align}
These equations are, respectively, known as
the thin film equation with linear mobility from lubrication theory \cite{gf-bertozzi},
and the quantum drift diffusion equation from semi-conductor modeling \cite{gf-juengel}.
The corresponding free energy functionals are the Dirichlet energy, and the Fisher information,
\begin{align}
  \label{gf:TFenergy}
  \nrg_\text{Dirichlet}(\rho) =\frac12\int_\Omega|\nabla\rho|^2\dd x
  \quad\text{and}\quad
  \nrg_\text{Fisher}(\rho) = \int_\Omega|\nabla\log\rho|^2\rho\dd x.
\end{align}

{\bf Notation and Basic concepts.-} For definiteness, let $\Omega\subset\setR^d$ be a bounded and convex domain throughout this chapter. Let us denote by $\mathcal{P}(\Omega)$ the set of probability measures with support on the open set $\Omega$. Further, recall the notation of push-forward of a density $\rho:\Omega\to\setR$ through a map $T:\Omega\to\Omega$:
\begin{align*}
  T\#\rho = \frac{\rho}{\det\dff T}\circ T^{-1}.
\end{align*}
There are several possible definitions of the euclidean transport distance $\wass$ between two probability measures, the most robust being the infimum in the Kantorovich problem:
\begin{align}
  \wass(\rho,\eta)^2 = \inf_{\gamma\in\Gamma(\rho,\eta)}\int_{\Omega\times\Omega}|x-y|^2\dd\gamma(x,y),
\end{align}
where $\Gamma(\rho,\eta)$ is the set of all couplings between $\rho$ and $\eta$,
that is,
\begin{align*}
  \Gamma(\rho,\eta) = \left\{\gamma\in\prb(\Omega\times\Omega)
  \,\middle|\,\forall A,B\subseteq\Omega:\gamma[A\times\Omega]=\rho[A],\,\gamma[\Omega\times B]=\eta[B] \right\}.
\end{align*}
For the Lagrangian approach discussed here, we shall use Monge's original definition,
which is equivalent to the one above at least if $\rho$ is absolutely continuous, see \cite{Brenier}:
\begin{align*}
  \wass(\rho,\eta)^2 = \inf_{T:\eta=T\#\rho}\int_\Omega |T(x)-x|^2\rho(x)\dd x.
\end{align*}
We refer to \cite{V2003,santambrogio2015optimal} for the basics of optimal transport concepts. There also exists an Eulerian approach to optimal transport introduced by Benamou and Brenier \cite{BB00}, we will postopone its dicussion to the last section where we will make use of it.

\subsection{Lagrangian formulation of \eqref{gf:0}}

The transport character of \eqref{gf:0} calls for a Lagrangian formulation of the dynamics.
We adopt the microscopic picture of particles that move along the vector field $\velo$,
which is influenced by themselves through the induced change of the macroscopic particle density $\rho$.
For definiteness, let $\Theta\subset\setR^d$ be a reference domain
and $\theta\in\prb(\Theta)$ be a reference probability density.
A canonical choice is $\Theta=\Omega$ and $\theta=u_0$, the initial condition,
but we shall discuss further possible choices below.
The Lagrangian map $X_{(\cdot)}:[0,T]\times\Theta\to\Omega$ is associated to a solution $\rho_t$ of \eqref{gf:0},
such that $t\mapsto X_t(\xi)\in\Omega$ describes the trajectory of a particle with label $\xi\in\Theta$.
That is,
\begin{align}
  \partial_tX_t = \velo_t\circ X_t.
\end{align}
Provided that $X_0$ is chosen such that $X_0\#\theta=\rho_0$, i.e., $X_0$ realizes the initial density,
then
\begin{align*}
  \rho_t=X_t\#\theta
\end{align*}
at any $t\in[0,T]$. Next we introduce
\begin{align*}
  \nrg^\#(X) := \nrg(X\#\theta),
\end{align*}
and use that for every diffeomorphism $\bar X:\Theta\to\Omega$, 
\begin{align*}
  \frac1 \theta\frac{\delta\nrg^\#}{\delta X}(\bar X)
  = \nabla\frac{\delta\nrg}{\delta\rho}(\bar X\#\theta)\circ\bar X,
\end{align*}
see Lemma \ref{gf:lem.cov} in the appendix.
Then the probability density $\rho_t$ can be completely eliminated from the evolution equation,
leading to a closed system for $X$ alone,
\begin{align}
  \label{gf:1}
  \partial_tX_t = -\frac1\theta\frac{\delta\nrg^\#}{\delta X}(X_t).
\end{align}
Note that equation \eqref{gf:0} and \eqref{gf:1} are equivalent for smooth solutions with everywhere-positive $\rho_t$.

\subsection{Two gradient flow structures}
We briefly recall the basic terminology of gradient flows:
on a Riemannian manifold $\mf$ with a local scalar product $\langle\cdot,\cdot\rangle_x$,
the gradient flow of a function $F\in C^2(\mf)$ is formed by solutions $x_{(\cdot)}:[0,T]\to\mf$
to
\begin{align}
  \label{gf:mf}
  \frac{d}{dt}x_t = - \grd_\mf F(x_t),
\end{align}
where the gradient $\grd_\mf F(x)$ of $F$ at $x\in\mf$
is the unique element $v\in\mathrm{T}_x\mf$ in $\mf$'s tangent space at $x$
with $\langle v,w\rangle_x = \dff F(x)[w]$ for all $w\in\mathrm T_x\mf$.
With this definition of the gradient, it is clear that any solution $x_{(\cdot)}$ to \eqref{gf:mf}
descends in $F$'s potential landscape ``as fast as possible'' in the sense that
a curve $x_{(\cdot)}:[0,T]\to\mf$ is a solution if and only if at each instance of time $t\in[0,T]$,
its tangent vector $dx_t/dt\in\mathrm T_x\mf$ minimizes
among all $v\in\mathrm T_x\mf$ the following expression:
\begin{align}
  \label{gf:gfvar}
  \frac12\langle v,v\rangle_{x_t} + \dff F(x_t)[v].
\end{align}
There are two ways in which \eqref{gf:0} can be considered as a gradient flow.

The first is obvious from the Lagrangian formulation \eqref{gf:1},
which is directly identified as an \emph{$L^2$-gradient flow of $\nrg^\#$ on the space of Lagrangian maps}.
Here the role of the manifold $\mf$ is played by $L^2_\theta(\Theta;\Omega)$,
the linear space of measurable maps $X:\Theta\to\Omega$,
whose tangent space at any point consists of square integrable vector fields $\velo:\Theta\to\setR^d$,
and is equipped with the scalar product
\begin{align*}
  \langle\velo,\welo\rangle_\theta = \int_\Theta\velo\cdot\welo\,\theta\dd\xi.
\end{align*}
While this $L^2$-gradient flow structure is easily understood,
the second structure introduced in the seminal paper by Otto \cite{O2001} is more subtle:
\eqref{gf:0} is a \emph{metric gradient flow of $\nrg$ in the $L^2$-Wasserstein distance}.
Here, the role of the manifold $\mf$ is played by the space of probability measures $\prb(\Omega)$.
The rigorous construction of  the Wasserstein tangent space at some $\rho\in\prb(\Omega)$
amounts to identifying --- via the continuity equation --- tangent vectors with elements in the closure
of all gradient vector fields $\nabla\varphi:\Omega\to\setR^d$ with $\varphi\in C^\infty_c(\Omega)$  
in the $\rho$-weighted $L^2$-norm.  
Less rigorously and more intuitively:
if $(\rho_t)_{t\in[0,T]}$ is some sufficiently regular curve in $\prb(\Omega)$,
then at each $t\in[0,T]$, there is an essentially unique gradient vector field $\velo_t=\nabla\varphi_t$
such that $\partial_t\rho_t+\dv(\rho_t\velo_t)=0$.
Inside that framework, the analogue of the local scalar product between two tangent vectors,
identified with $\nabla\varphi$ and $\nabla\psi$, respectively, amounts to
\begin{align}
  \label{gf:ottocalculus}
  \langle\nabla\varphi,\nabla\psi\rangle_\rho 
  := \int_\Omega \nabla\varphi\cdot\nabla\psi\,\rho\dd x.
\end{align}
In this context, the $L^2$-Wasserstein distance $\wass(\rho^0,\rho^1)$ between $\rho^0,\rho^1\in\prb(\Omega)$
can be introduced as shortest connecting curve $(\rho^s)_{0\le s\le1}$ from $\rho^0$ to $\rho^1$,
\begin{align}
  \label{gf:W2aslength}
  \wass(\rho^0,\rho^1) = \inf_{(\rho^s)_{0\le s\le1}}
  \left\{\int_0^1\langle\nabla\psi^s,\nabla\psi^s\rangle_{\rho^s}\dd s\,
  \middle|\,\partial_s\rho^s+\dv(\rho^s\nabla\psi^s)=0\right\}.
\end{align}

With this dictionary at hand, it is now straight-forward to conclude
that \eqref{gf:0} is the analogue of \eqref{gf:mf},
and in particular, $\velo_t$ is identified with the gradient of $\nrg$ at $\rho_t$
with respect to the local scalar product \eqref{gf:ottocalculus}.
In analogy to the situation on the Riemannian manifold,
we define the gradient in the Wasserstein metric (with a certain abuse of notation) by
\begin{align}
  \label{gf:Wgrad}
  \grd_\wass\nrg = \nabla\frac{\delta\nrg}{\delta\rho},
\end{align}
so that \eqref{gf:0} becomes
\begin{align*}
  \partial_t\rho_t = \dv\big(\rho_t\grd_\wass\nrg(\rho_t)\big).
\end{align*}
We stress that this discussion is very formal.
The introduction of the $L^2$-Wasserstein distance $\wass$ via \eqref{gf:W2aslength} is full of technical subtleties, see \cite{BB00}.
Moreover, the functionals $\nrg$ of interest are far from being differentiable,
but typically just lower semi-continuous,
which makes it impossible to define a gradient in a simple way as in \eqref{gf:Wgrad}.

There are fully rigorous approaches to understanding solutions to \eqref{gf:0} as being of steepest descent,
the most prominent being the theory of metric gradient flows developed by Ambrosio et al \cite{gf-AGS}.
There, one completely avoids the scalar product \eqref{gf:ottocalculus}
and defines $\wass$ as a global metric on $\prb(\Omega)$.
Accordingly, the metric gradient flow is \emph{not} formulated in differential terms,
but instead is characterized by variational principles that are formally related to \eqref{gf:gfvar} above,
that is, to consider curves $t\mapsto\rho_t$ in $\prb(\Omega)$
whose motion --- now measured in $\wass$ --- is such that it decreases the functional $\nrg$ as fast as possible.
This can be formulated in robust variational ways,
like the energy dissipation equality, or the evolutionary variational inequalities.

Unfortunately, this approach is very abstract,
and it has surprising limitations in view of applicability to concrete evolution equations of type \eqref{gf:0}.
While the non-linear Fokker-Planck equation \eqref{gf:FP} fits well into the framework of \cite{gf-AGS},
this is not the case for the fourth order equations \eqref{gf:TF}.
Therefore, we shall adopt a more practical view on Wasserstein gradient flows here,
which is centered around the question:
``What are the properties implied on solutions to \eqref{gf:0} by the theory,
and how can they be used for the design of numerical schemes?''

\section{Benefit from the gradient flow structures}
As indicated above, we shall not go into details of the theory of metric gradient flows
but only summarize results and techniques that are of interest for the design and analysis of numerical schemes.
 
\subsection{Existence of solutions by minimizing movements}
Knowing that \eqref{gf:0} is (at least formally) a gradient flow in the Wasserstein metric for a reasonable functional $\nrg$,
one almost automatically obtains the existence of curves $\rho_{(\cdot)}$ in $\prb(\Omega)$
that realize the principle of steepest descent indicated in \eqref{gf:gfvar}.
The general approach to their construction goes
via the celebrated variational form of the implicit Euler discretization in time,
commonly referred to as minimizing movement scheme.
More specifically, for a given time step size $\dt>0$,
a sequence $(\rho_\dt^n)_{n=0}^\infty$ of $\rho_\dt^n\in\prb(\Omega)$
such that $\rho_\dt^n$ approximates the density $\rho(n\dt)$ of the true solution at time $t=n\tau$
are obtained as follows:
one starts from the initial condition $\rho_\dt^0:=\rho_0$,
and then one defines each $\rho_\dt^n$ for $n=1,2,\ldots$ inductively as minimizer of
\begin{align}
  \label{gf:mm}
  \rho\mapsto\frac1{2\dt}\wass(\rho,\rho_\dt^{n-1})^2 + \nrg(\rho).
\end{align}
Note that \eqref{gf:mm} can be formally derived
by integrating \eqref{gf:gfvar} in time from $t=(n-1)\tau$ to $t=n\tau$ along a solution $\rho_{(\cdot)}$,
and approximating the integral over the metric term by the distance.
Unqiue solvability of \eqref{gf:mm} follows under the typical hypotheses from calculus of variations.

By abstract arguments, the time-interpolated $\rho_\dt^n$
have accumulation points in the space of continuous curves on $\prb(\Omega)$ for $\dt\to0$.
The difficulty is then to show that these limit curves are indeed solutions to \eqref{gf:0} in some sense.
This has been shown for the Fokker-Planck equation \eqref{gf:FP} and for the thin film and QDD equations \eqref{gf:TF},
see \cite{gf-JKO,gf-Otto2,gf-Otto1,gf-BCC,gf-BCC2,gf-GST,gf-MMS,gf-CFetc}. This method of approximation in time is so reliable that it is the basis for essentially all numerical schemes for gradient flows.
We mention that variational discretizations in time of higher order have been developed recently \cite{gf-LT,gf-MP,gf-Plazotta},
but are not yet widely used.
 
\subsection{Long time asymptotics from displacement convexity}
Some functionals $\nrg$ of interest
happen to be $\lambda$-uniformly displacement convex in the sense of McCann \cite{gf-McCann}:
that is, there is a $\lambda\in\setR$ such that 
for any ``reasonable'' unit speed geodesic $(\eta_s)_{s\in[0,\sigma]}$ in $(\prb(\Omega),\wass)$,
the real function $s\mapsto\nrg(\eta_s)$ has second derivative bounded below by $\lambda$.
This is true for instance for the entropy functional $\ent_{h,V,W}$ under the hypotheses that
$h$ satisfies the McCann condition:
\begin{enumerate}
\item $h$ is convex with $h(0)=0$, and $s\mapsto s^{d}h(s^{-d})$ is convex and non-increasing,
\item $W$ is convex, and
\item $V$ is $\lambda$-uniformly convex, that is $\nabla^2V\ge\lambda\eins$.
\end{enumerate}
For Wasserstein gradient flows of such functionals,
one obtains interesting consequences on the long-time asymptotics of solutions $\rho_{(\cdot)}$.

A first implication of $\lambda$-uniform displacement convexity is $\lambda$-uniform contractivity of the flow:
for any two solutions $\rho_{(\cdot)}$ and $\eta_{(\cdot)}$ the map
\begin{align}
  \label{gf:contract}
  t\mapsto e^{\lambda t}\wass(\rho_t,\eta_t)
\end{align}
is non-increasing in time.
Note that for $\lambda>0$, this implies mutual attraction of solutions at exponential rate in time,
for $\lambda<0$, it estimates the speed of divergence from each other.
In any case, monotonicity of \eqref{gf:contract} implies uniqueness.
If $\lambda>0$, then it follows from contractivity that \eqref{gf:0} has a unique stationary solution $\rho_*$,
which coincides with the unique minimizer of $\nrg$,
and any solution $\rho_{(\cdot)}$ to \eqref{gf:0} converges to $\rho_*$ exponentially fast:
\begin{align}
  \label{gf:decay}
  t\mapsto e^{\lambda t}\wass(\rho_t,\rho_*)
  \quad\text{and}\quad
  t\mapsto e^{\lambda t}\big(\nrg(\rho_t)-\nrg(\rho_*)\big)
\end{align}
are non-increasing in $t\ge0$.
The estimates in \eqref{gf:decay} have been the key to study the long-time asymptotics of various PDEs of type \eqref{gf:0},
see e.g. \cite{gf-Otto1,CMV2003,cmcv-06,gf-MMS,gf-CFetc,gf-BCC2}.

Preserving convexity is one of the central interests in designing numerical methods for \eqref{gf:0}.
If $\prb(\Omega)$ is approximated by a finite-dimensional metric space,
and $\nrg$ is discretized thereon such that it has the same modulus $\lambda$ of convexity,
then the discretized gradient flow is automatically asymptotic preserving in the sense that
solutions share the monotonicities \eqref{gf:contract} and \eqref{gf:decay}.
Moreover, from the computational perspective,
a very useful consequence of $\lambda$-uniform displacement convexity 
is that the minimization problem \eqref{gf:mm} is uniformly convex as well,
with modulus $\lambda':=\frac1\tau+\lambda$.


\subsection{Polyconvexity from displacement convexity}
A benefit from the gradient flow structure of \eqref{gf:0} in $\wass$ 
is that the evolution \eqref{gf:1} for the Lagrangian map $X$ is a gradient flow as well:
one replaces the space of probability measures with the intricate Wasserstein distance
by a space of maps, equipped with the much easier $L^2$-structure.
The construction of solutions via minimizing movements carries over from \eqref{gf:mm}:
a sequence of maps $(X_\tau^n)_{n=0}^\infty$ is inductively obtained by minimizing
\begin{align}
  \label{gf:mmL}
  \frac1{2\dt}\|X-X_\dt^{n-1}\|_{L^2_\theta}^2 + \nrg^\#(X).
\end{align}
In contrast to \eqref{gf:mm}, this problem is much easier to solve in practice,
since it does not involve the calculation of the Wasserstein distance,
but only of an $L^2$-norm.

There is, however, a price to pay.
An obvious drawback is that the transformed functionals $\nrg^\#$
typically attain a much more complicated form than their respective originals $\nrg$.
For instance, in the comparatively easy case of the relative entropy $\ent_{h,V,W}$ from \eqref{gf:FPenergy},
one obtains
\begin{align}
  \label{gf:FPehash}
  \begin{split}
    \ent_{h,V,W}^\#(X)
    &= \int_\Theta\left[
      h^\#\left(\frac{\det\dff X}\theta\right)
      + V(X)\right]\theta\dd\xi \\
    &\quad + \int_\Theta\int_\Theta W\big(X(\xi)-X(\xi')\big)\theta(\xi)\theta(\xi')\dd\xi\dd\xi',
  \end{split}
\end{align}
with the definition
\begin{align}
  \label{gf:hhash}
  h^\#(s) = s h(s^{-1}).
\end{align}
Another, more subtle difficulty arises in any space dimension $d>1$:
the correspondence $X\mapsto X\#\theta$ is \emph{not} an isometry between
the $L^2$-distance on injective monotone maps and the Wasserstein distance on densities.
In fact, that correspondence is highly non-unique
--- for any given sufficiently regular $\rho$,
there are infinitely many genuinely different maps $X$ such that $X\#\theta=\rho$ ---  
and there is no ``universal normalization'' of the $X$'s
such that $\wass(X_0\#\theta,X_1\#\theta)=\|X_0-X_1\|_{L^2(\theta)}$ would be true in general.
In particular, the linear interpolation between $X_0$ and $X_1$ has typically little to do
with the Wasserstein geodesic connecting $X_0\#\theta$ to $X_1\#\theta$.
Therefore, displacement convexity of $\nrg$ does usually not imply any flat convexity of $\nrg^\#$,
and contractivity of the gradient flow \eqref{gf:0} with respect to $\wass$
does not imply contractivity of \eqref{gf:1} with respect to $L^2$.
What remains from displacement convexity is polyconvexity: 
in the situation above, the functional in \eqref{gf:FPehash} is indeed polyconvex
if the corresponding $\ent_{h,V,W}$ is $\lambda$-uniformly displacement convex with $\lambda\ge0$. 
We remark that the correspondence between \eqref{gf:0} and \eqref{gf:1} has first been pointed out
--- at least in the special case of the Fokker-Planck equation \eqref{gf:FP} --- by Evans et al \cite{gf-EGS}, where
they also propose to solve \eqref{gf:1} by the implicit time discretization \eqref{gf:mmL},
based on the polyconvexity of the functional $\nrg^\#$,
It was then shown later \cite{gf-ALS} that, at least for $\nrg=\ent_{h,0,0}$,
there is indeed a one-to-one correspondence between the variational problems \eqref{gf:mm} and \eqref{gf:mmL},
i.e., under suitable hypotheses on the initial condition $\rho_0$,
the discrete iterates coincide, $\rho_\tau^n=X_\tau^n\#\rho_0$. This fact was later used for constructing maps 
joining particular densites with given Jacobians, see \cite{CL}.
\medskip

\textbf{In summary:}
The gradient flow structures lead to natural time-discretizations of \eqref{gf:0}
that are in variational form, see either equation \eqref{gf:mm} or equation \eqref{gf:mmL}.
These variational problems are even strictly geodesically convex or polyconvex, respectively,
if $\nrg$ happens to be $\lambda$-uniformly displacement convex in the sense of McCann.
Preserving that convexity also under spatial discretization leads to schemes
that replicate contractivity \eqref{gf:contract} and convergence to equilibrium \eqref{gf:decay},
thus reproducing the correct long time asymptotics.
Moreover, discretizing the Lagrangian form \eqref{gf:1} of the dynamics
automatically ensures properties like mass conservation and non-negativity.


\section{1D Wasserstein using the inverse distribution functions}
In space dimension $d=1$ the conceptual difference between
performing a minimization in \eqref{gf:mm} on a discretized set of Lagrangian maps,
and 
discretizing the Lagrangian equations directly
is little.
The reason is the isometry between the Wasserstein space over an interval $I=[a,b]\subset\setR$,
and the $L^2$-space of inverse distribution functions with values in $I$.

\subsection{The inverse distribution function}
We recall the basic definitions and relations:
\begin{itemize}
\item To each $\rho\in\prb(I)$, one associates the cumulative distribution function $F_\rho:I\to\Theta$ via $F_\rho(x)=\rho[[a,x]]$,
  where $\Theta=[0,1]$.
  By outer regularity, $F_\rho$ is non-decreasing and \cadlag.
  It is strictly increasing on the support of $\rho$, and it is continuous if $\rho$ is an absolutely continuous measure.
\item By the usual construction, one obtains a unique non-decreasing \cadlag\ right inverse $X_\rho:\Theta\to I$ of $F_\rho$,
  called $\rho$'s \emph{inverse distribution function}.
  $X_\rho$ is a genuine inverse if $\rho$ is absolutely continuous and has support $I$.
\item The definition directly implies that, with $\theta$ being the Lebesgue measure on $\Theta=[0,1]$,
  \begin{align}
    \label{gf:X2rho}
    \rho = X_\rho\#\theta,
  \end{align}
  and $X_\rho$ is the only non-decreasing \cadlag\ function $X:\Theta\to I$ with that property.
  Consequently, if $\rho$ is absolutely continuous with an everywhere positive and continuous density function,
  then $X_\rho$ is continuously differentiable with
  \begin{align}
   \rho\circ X_\rho = \frac1{\partial_\xi X_\rho}.
  \end{align}
  This relation generalizes in the obvious way when $\rho$'s density has isolated points of discontinuity.
\item The association $\rho\mapsto X_\rho$ is an isometry in the following sense:
  \begin{align*}
    \wass (\rho,\eta) = \|X_\rho-X_\eta\|_{L^2([0,1])}.
  \end{align*}
  Particularly, the space $(\prb(I),\wass)$ is flat,
  and Wasserstein geodesics are given through linear interpolation of the respective inverse distribution functions.
\end{itemize}
The following equivalence is obvious:
if a sequence $(\rho_\dt^n)_{n=0}^\infty$ of $\rho_\dt^n\in\prb(I)$ is minimal for \eqref{gf:mm},
then the sequence $(X_\dt^n)_{n=0}^\infty$ of respective inverse distribution functions $X_\tau^n:=X_{\rho_\tau^n}$ is minimal in \eqref{gf:mmL};
and if a sequence $(X_\dt^n)_{n=0}^\infty$ of non-decreasing \cadlag functions $X_\dt^n:\Theta\to I$ is minimal in \eqref{gf:mmL}, 
then the sequence $(\rho_\dt^n)_{n=0}^\infty$ of respective densities $\rho_\dt^n:=X_\dt^n\#\theta$ is minimal in \eqref{gf:mm}.
Moreover, if $\nrg$ is $\lambda$-uniformly displacement convex, then the minimization problem \eqref{gf:mmL} is convex of modulus $\frac1\tau+\lambda$. In fact contraction estimates in one dimension can be obtained in this formulation \cite{CT2004,LT2004,gf-BCC}. This formulation can also be used to obtain well-posedness of solutions in the theoretical setting \cite{CHR} even when blow-up of the densities can occur.

\subsection{Discretization}
Lagrangian numerical schemes for solution of the one-dimensional Fokker-Planck, thin film and QDD equations
have been devised by various authors, see e.g. \cite{gf-GT1,gf-GT2,gf-BCC,gf-WW,gf-CN,CRW2016}.
Below, we review the ansatz made in \cite{gf-MOfp,gf-MOtf,gf-MOdlss,gf-Osberger}.

As ansatz space $\ansatz$ for the inverse distribution functions $X$,
we choose the continuous and strictly increasing functions $X:\Theta\to I$
that are piecewise linear with respect to a given partition $\bfxi=(\xi_k)_{k=0}^K$ of $\Theta$,
\begin{align*}
  0=\xi_0 < \xi_1 < \cdots < x_K=1,
\end{align*}
More explicitly, elements $X\in\ansatz$ are in one-to-one correspondence to partitions $\bfx=(x_k)_{k=0}^K$ of $I$
with
\begin{align*}
  a=x_0<x_1<\cdots<x_K=b,
\end{align*}
by means of
\begin{align*}
  X = \xtoX[\bfx] := \sum_{k=0}^K x_k\phi_k,
\end{align*}
where the $\phi_k:\Theta\to\setR$ are the usual hat functions,
with $\phi_k(\xi_k)=1$, and $\phi_k(\xi_\ell)=0$ for $\ell\neq k$.
The associated probability density $\rho\in\prb(I)$ with $X_\rho=X$ is piecewise constant,
given by
\begin{align*}
  \rho = \xtorho[\bfx] := \sum_{k=1}^K \rho_k\indy_{(x_{k-1},x_k)},
  \quad\text{with respective values}\quad
  \rho_k:=\frac{x_k-x_{k-1}}{\xi_k-\xi_{k-1}}.
\end{align*}
For the $L^2$-norm of a difference $X-X'$ for $X,X'\in\ansatz$, one obtains
\begin{align*}
  \|X-X'\|_{L^2([0,1])}^2 = \langle \bfx-\bfx',\bfx-\bfx'\rangle_{L^2_\bfxi}
  \quad\text{with}\quad
  \langle v,v\rangle_{L^2_\bfxi} := v^TA_\bfxi v
\end{align*}
for each $v\in\setR^{K+1}$, where $A_\bfxi\in\setR^{(K+1)\times(K+1)}$ is the well-known tri-diagonal stiffness matrix.
For practical purposes, $A_\bfxi$ can even be replaced by its canonical diagonal approximation without significant harm for the numerical results.

Now let $\nrg^\#_\bfxi$ be some approximation of $\nrg^\#$ on partitions $\bfx\in\setR^{K+1}$ of $I$. 
For instance, if $\nrg$ depends on $\rho$, but not on its derivatives,
then one may choose $\nrg^\#_\bfxi(\bfx) = \nrg^\#(\xtoX[\bfx])$.
The gradient flow of $\nrg^\#_\bfxi$ with respect to the inner product $\langle\cdot,\cdot\rangle_{L^2_\bfxi}$
is 
\begin{align}
  \label{gf:dWgrad}
  -\dot\bfx_t = \grd_\bfxi\nrg^\# := A_\bfxi^{-1}\left(\frac{\partial}{\partial x_k} \nrg^\#_\bfxi\right)_{k=0}^K.
\end{align}

\subsection{Discretizing the Fokker-Planck equation}
We consider $\nrg$ of the form \eqref{gf:FPenergy}.
Here one can directly evaluate \eqref{gf:mmL} on the ansatz space $\ansatz$,
which yields, recalling \eqref{gf:FPehash}:
\begin{align*}
  \ent_{h,V,W}^\#(\xtoX[\bfx]) 
  &= \sum_{k=1}^K(\xi_k-\xi_{k-1})\left[h(z_k) + \aint_{x_{k-1}}^{x_k}V(x)\dd x\right] \\
  &\quad + \sum_{k,\ell=1}^K(\xi_k-\xi_{k-1})(\xi_\ell-\xi_{\ell-1})\aint_{x_{k-1}}^{x_k}\aint_{x_{\ell-1}}^{x_\ell}W(x-y)\dd x\dd y
\end{align*}
For practical purposes, a sufficiently precise approximation is
\begin{align*}
  \big[\ent_{h,V,W}^\#\big]_\bfxi(\bfx)
  &:= \sum_{k=1}^K(\xi_k-\xi_{k-1})\left[h(z_k) + V\left(\frac{x_k+x_{k-1}}2\right)\right] \\
  &\quad + \sum_{k,\ell=1}^K(\xi_k-\xi_{k-1})(\xi_\ell-\xi_{\ell-1})W\left(\frac{x_k-x_\ell+x_{k-1}-x_{\ell-1}}2\right).
\end{align*}
We remark that this approximation preserves the modulus of convexity.

On basis of this,
a fully discrete Lagrangian scheme for solution of \eqref{gf:FP} has been developed in \cite{gf-MOfp,gf-MS},
by additionally discretizing \eqref{gf:dWgrad} in time via the implicit Euler method,
\begin{align}
  \label{gf:001}
  \frac{\bfx^n-\bfx^{n-1}}{\dt} = -\grd_\xi\big[\ent_{h,v,W}^\#\big]_\xi(\bfx^n).
\end{align}
The fully discrete solutions are well-defined since the $\bfx^n$ can be obtained inductively by solving minimization problems.
In \cite{gf-MOfp,gf-MS}, convergence of the scheme has been shown:
\begin{theorem}
  \label{gf:thm.fpconverge}
  Consider a sequence of spatial meshes $\bfxi^{(j)}$ and time steps $\dt^{(j)}$,
  such that $\max_k(\xi^{(j)}_k-\xi^{(j)}_{k-1})\to0$ and $\dt^{(k)}\to0$ as $k\to\infty$,
  while $\max_k(\xi^{(j)}_k-\xi^{(j)}_{k-1})/\min_k(\xi^{(j)}_k-\xi^{(j)}_{k-1})$ remains bounded.
  Let initial data $\bfx^{(j)}_0$ be given such that $\xtorho[\bfx^{(j)}_0]\to\rho_0$ in $L^1(I)$,
  and $\big[\ent_{h,V,W}^\#\big]_{\bfxi^{(j)}}(\bfx^{(j)}_0)$ remains bounded.

  Then the piecewise constant interpolations $\bar\rho^{(j)}:[0,T]\to\prb(I)$,
  obtained from the fully discrete solutions to \eqref{gf:001}
  via $\rho^{(j)}(t,\cdot)=\xtorho\big([\bfx^{(j)}]^n\big)$ for $t\in((n-1)\dt,n\dt)$
  converge strongly in $L^1([0,T]\times I)$ to the unique weak solution of \eqref{gf:FP}
  with initial datum $\rho_0$.
\end{theorem}
The original statement in \cite{gf-MOfp} contained a CFL condition;
it was later shown \cite{gf-Odiss} that it is not necessary. Furthermore they proved that 
the discrete solutions satisfy the same maximum and minimum principles as well as the correct contraction estimate \eqref{gf:contract}.

\subsection{Discretizing fourth order equations}
For the functionals $\nrg$ in \eqref{gf:TFenergy} that contain gradients of $\rho$,
the respective
\begin{align}
  \label{gf:FDtilde}
  \nrg_\text{Fisher}^\#(X) = \int_0^1\left|\partial_\xi\left(\frac1{\partial_\xi X}\right)\right|^2\dd\xi
  \quad\text{and}\quad
  \nrg_\text{Dirichlet}^\#(X) = \int_0^1\left|\partial_\xi\left(\frac1{\sqrt{\partial_\xi X}}\right)\right|^2\dd\xi
\end{align}
contain second derivatives of $X$ and hence cannot be directly evaluated on $\ansatz$.
Appropriate surrogates are needed that can be defined on piecewise linear inverse distribution functions $X$'s.

The ansatz that has been taken in \cite{gf-MOtf,gf-MOdlss} is based on very particular relations
of the Dirichlet energy and Fisher information to entropy functional,
which we state here in more general terms:
consider
\begin{align*}
  \ent_1(\rho) = \int_\Omega\rho\log\rho\dd x
  \quad\text{and}\quad
  \ent_{2}(\rho) = \int_\Omega \rho^{2}\dd x,
\end{align*}
respectively.
Then, recalling the definition of $\grd$ in \eqref{gf:Wgrad},
one has at any given positive and smooth $\rho\in\prb(\Omega)$:
\begin{align*}
  \nrg_\text{Fisher}(\rho)
  &= \big\langle\grd\ent_1(\rho),\grd\ent_1(\rho)\big\rangle_\rho\\
  \nrg_\text{Dirichlet}(\rho)
    &= \big\langle\grd\ent_1(\rho),\grd\ent_{2}(\rho)\big\rangle_\rho.
\end{align*}
This observation motivates to \emph{define} the discrete surrogates for the functionals in \eqref{gf:FDtilde}
by means of these relations.
We recall the definition of the discrete gradient from \eqref{gf:dWgrad}, that is
\begin{align}
  \label{gf:dTFenergy}
  \big(\nrg_\text{Fisher}\big)^\#_\bfxi &:= \langle\grd_\bfxi\ent_1^\#(\xtoX),\grd_\bfxi\ent_1^\#(\xtoX)\rangle_{L^2_\bfxi},\\
  \big(\nrg_\text{Dirichlet}\big)^\#_\bfxi &:= \langle\grd_\bfxi\ent_1^\#(\xtoX),\grd_\bfxi\ent_{2}^\#(\xtoX)\rangle_{L^2_\bfxi}.  
\end{align}
In analogy to \eqref{gf:001},
fully discrete Lagrangian schemes for solution of the QDD and the thin film equation have been developed,
\begin{align}
  \label{gf:002}
  \frac{\bfx^n-\bfx^{n-1}}{\dt} &= -\grd_\xi\big[\nrg_{Dirichlet}^\#\big]_\xi(\bfx^n),\\
  \frac{\bfx^n-\bfx^{n-1}}{\dt} &= -\grd_\xi\big[\nrg_{Fisher}^\#\big]_\xi(\bfx^n),
\end{align}
respectively.
In the convergence analysis, the precise form of the discretized Dirichlet energy and Fisher information play a decisive role.
Namely, it follows from abstract considerations that $\ent_1^\#(\xtoX)$ is a Lyapunov functional for both discretizations,
and that its dissipation provides the necessary a priori estimates.
We state the final result for the QDD equation; the statement for the thin film equation is completely analogous.
\begin{theorem}
  Assume the same hypotheses as in Theorem \ref{gf:thm.fpconverge} are satisfied,
  only that the spatial discretizations $\bfxi^{(j)}$ are equi-distant,
  and that $\big[\nrg_{Fisher}^\#\big]_{\bfxi^{(j)}}(\bfx^{(j)}_0)$ remains bounded.
  Construct the piecewise constant in time and space approximations $\rho^{(j)}:[0,T]\to\prb(I)$ as before
  from the fully discrete solutions to \eqref{gf:002}.
  Then $\rho^{(j)}$ converges uniformly on $[0,T]\times I$ to a weak solution of the QDD equation in \eqref{gf:TF}
  with initial datum $\rho_0$.
\end{theorem}
A variant of the definitions in \eqref{gf:dTFenergy} has been used by Osberger \cite{gf-Osberger}
to approximate solutions to the free boundary problems, i.e., $I=\setR$ and $\rho_0$ is compactly supported.
There, it is shown that the discrete solution converges at the expected rate towards self-similarity.


\section{Multi-D: first discretize, then optimize}
In this section, we focus on Lagrangian discretizations of \eqref{gf:0} on the two-dimensional box $\Omega=[0,1]^2$
for Fokker-Planck equations without interaction term,
\begin{align}\label{fpnonlinear}
  \partial_t\rho_t = \Delta\Phi(\rho_t) + \dv(\rho_t\,\nabla V),
\end{align}
with $V\in C^2(\Omega)$.
The three methods that we review extend to more general convex domains
and to higher space dimensions $d>2$ without any conceptually new ideas;
the choice $\Omega=[0,1]^2$ is mainly made to enhance readability.
The question of generalization of the methods to flows of type \eqref{gf:0} other than \eqref{fpnonlinear} is more subtle,
and will be discussed for each methods individually.

\subsection{Moving triangle meshes}
First, we review the ansatz made in \cite{gf-CDMM}, which is the most straight-forward generalization
of the approach via the inverse distribution function to a Lagrangian method in multiple dimensions.
We assume that on the reference domain $\Theta:=\Omega=[0,1]^2$,
a triangulation with $K$ verticies $\xi_k$ is given, which we denote (by abuse of notation) by $\bfxi$.
We write $(k,\ell,m)\in\bfxi$ if $\xi_k$, $\xi_\ell$ and $\xi_m$ are --- in that order ---
the vertices of a positively oriented triangle in $\bfxi$,
and we denote by $\Delta(\xi_k,\xi_\ell,\xi_m)\subset\Theta$ the geometric domain of the triangle with these corners.
Our ansatz space $\ansatz$ for the Lagrangian maps is that of continuous $X:\Theta\to\Omega$
that are piecewise linear on the triangles of $\bfxi$.
That is, $X$ maps each reference triangle in $\Theta$ linearly to an image triangle in $\Omega$.
We further assume that $X$ moves vertices on the boundaries only laterally,
and consequently, it fixes the corners of $\Omega$.
In analogy to the one-dimensional situation,
there is a one-to-one correspondence between $X\in\ansatz$
and the vector $\bfx=(x_k)_{k=1}^K$ of images $x_k=X(\xi_k)$ of the vertices $\xi_k$ of $\bfxi$,
which is again realized by a map $\xtoX:(\setR^2)^K\to\ansatz$ with
\begin{align}\label{e:fem}
  \xtoX[\bfx] = \sum_{k=1}^Kx_k\phi_k,
\end{align}
with $\phi_k:\Theta\to\setR$ the piecewise linear hat function with $\phi_k(\xi_k)=1$ and $\phi_k(\xi_\ell)=0$ for $k\neq\ell$.
Note that $\xtoX[\bfx]$ is globally injective if and only if
the image triangles formed by $\bfx$ (with the combinatorics from $\bfxi$) do not overlap;
this is the natural two-dimensional generalization of the $\bfx$'s monotonicity in $d=1$.

For simplicity, we assume that the reference measure $\theta$ is the Lebesgue measure.
It follows that the image density $\rho=\xtorho[\bfx]$ is again piecewise constant:
\begin{align*}
  \xtorho[\bfx] &= \sum_{(\xi_k,\xi_\ell,\xi_m)\in\bfxi} \rho_{k,\ell,m}\indy_{\Delta(x_k,x_\ell,x_m)}
  \quad \text{with} \\
  \rho_{k,\ell,m } &= \frac{|\Delta(\xi_k,\xi_\ell,\xi_m)|}{|\Delta(x_k,x_\ell,x_m)|}
  = \frac{\det(x_\ell-x_k|x_m-x_k)}{\det(\xi_\ell-\xi_k|\xi_m-\xi_k)}.
\end{align*}
The induced scalar product amounts to
\begin{align*}
  \langle \velo,\welo\rangle_{L^2_\bfxi} = \velo^TA_\bfxi\welo,
\end{align*}
where the $(k,\ell)$-entry of $A_\bfxi$ is given by $\int_\Theta \phi_k(\xi)\phi_\ell(\xi)\dd\xi$.
For the induced energy, we obtain
\begin{align*}
  \ent_{h,V,0}^\#(\xtoX[\bfx])
  = \sum_{(k,\ell,m)\in\bfxi}|\Delta(\xi_k,\xi_\ell,\xi_m)|\left[
  \frac{h(\rho_{k,\ell,m})}{\rho_{k,\ell,m}}+\aint_{\Delta(x_k,x_\ell,x_m)}V(x)\dd x\right],
\end{align*}
for which we use the reasonable approximation (recall $h^\#(s)=sh(1/s)$),
\begin{align*}
  \big[\ent_{h,V,0}^\#\big]_\bfxi(\bfx)
  = \sum_{(k,\ell,m)\in\bfxi}|\Delta(\xi_k,\xi_\ell,\xi_m)| & \left[
  h^\#\left(\frac{\det(x_\ell-x_k|x_m-x_k)}{\det(\xi_\ell-\xi_k|\xi_m-\xi_k)}\right) \right.\\
  &\left. \qquad +V\left(\frac{x_k+x_\ell+x_m}3\right)\right],  
\end{align*}
The full discretization is again obtained by using the implicit Euler discretization in time,
\begin{align}
  \label{gf:002}
  -\frac{\bfx^n-\bfx^{n-1}}{\dt}
  = \grd_\bfxi  \big[\ent_{h,V,0}^\#\big]_\bfxi(\bfx^n)
  := A_\bfxi^{-1}\left(\frac{\partial \big[\ent_{h,V,0}^\#\big]_\bfxi}{\partial x_k}(\bfx^n)\right)_{k=1}^K.
\end{align}
Naturally, the explicit calulation of the gradient is more involved than in $d=1$,
already because of the combinatorics from the triangulation.
The most important ingredient is this:
for any $\omega\in\setR^2$,
\begin{align*}
  g_{k,\ell,m}[\omega]&:=\omega\cdot\frac{\partial}{\partial x_k}h^\#\left(\frac{\det(x_\ell-x_k|x_m-x_k)}{\det(\xi_\ell-\xi_k|\xi_m-\xi_k)}\right) \\
  &= (h^\#)' \left(\frac{\det(x_\ell-x_k|x_m-x_k)}{\det(\xi_\ell-\xi_k|\xi_m-\xi_k)}\right)
  \frac{\tr\left[\cof(x_\ell-x_k|x_m-x_k)^T\,(-\omega|-\omega)\right]}{\det(\xi_\ell-\xi_k|\xi_m-\xi_k)} \\
  &=\frac{\Phi(\rho_{k,\ell,m}}{|\Delta(\xi_k,\xi_\ell,\xi_m)|}
    (1\ 1)\cof(x_\ell-x_k|x_m-x_k)^T\omega.
\end{align*}
Now using that for any regular $M\in\setR^{2\times2}$, one has
\begin{align*}
  \cof M = JAJ^T
  \quad\text{where}\quad
  J=\begin{pmatrix} 0 & -1 \\ 1 & 0 \end{pmatrix},
\end{align*}
one simplifies this further to obtain
\begin{align*}
  g_{k,\ell,m}[\omega]=\frac{\Phi(\rho_{k,\ell,m}}{|\Delta(\xi_k,\xi_\ell,\xi_m)|}\big[J(x_\ell-x_m)\big]\cdot\omega.
\end{align*}
To obtain the $k$th component of the gradient $\grd_\bfxi  \big[\ent_{h,V,0}^\#\big]_\bfxi(\bfx^n)$,
one thus needs to add up the expressions
\begin{align*}
  \Phi(\rho_{k,\ell,m})\,J(x_\ell-x_m) + \frac{|\Delta(\xi_k,\xi_\ell,\xi_m)|}3\nabla V\left(\frac{x_k+x_\ell+x_m}3\right)
\end{align*}
over all $k$ and $\ell$ such that $(k,\ell,m)\in\bfxi$.

The main results of \cite{gf-CDMM} are:
\begin{itemize}
\item The numerical scheme \eqref{gf:002} performs well in a variety of experiments.
\item The time stepping \eqref{gf:002} is variational, but the minimization problem is not convex.
\item The scheme is consistent with \eqref{gf:1} at first order as long as the triangulation induced by $\bfx$
  remains approximately perfectly hexagonal.
  As time evolves and mesh deformations becomes significantly non-symmetric, consistency will indeed fail.
\end{itemize}
The mentioned loss of convexity in \eqref{gf:002} calls for a remark.
In \cite{gf-CDMM}, we present a quite general construction of a curve in the space of Lagrangian maps along which convexity fails.
However, these example curves are related to \emph{rotations} of triangles in the plane;
in fact, it follows from the calculations in \cite[Appendix C]{gf-CDMM} that along curves in the space of Lagrangian maps
which are locally given by matrices with real eigenvalues, convexity does hold.
This explains why we have never observed any difficulties related to this mild form of non-convexity in the numerical simulations:
rotation of triangles is a very unlikely phenomenon to appear in the minimization process,
since they cost kinetic energy, but hardly change the potential energy.
\begin{figure}[ht!]
  \includegraphics[width=0.45\textwidth]{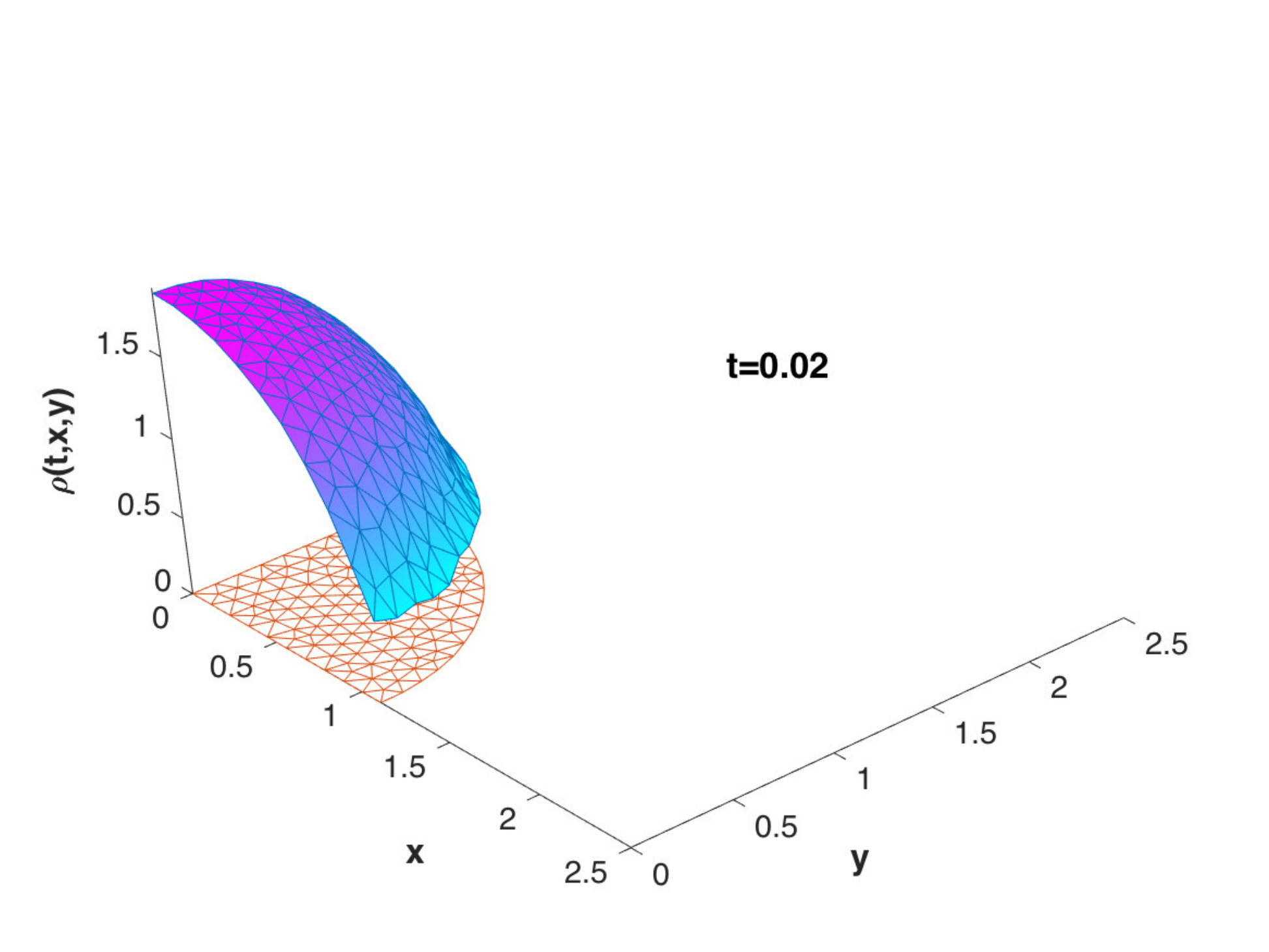}%
  \includegraphics[width=0.45\textwidth]{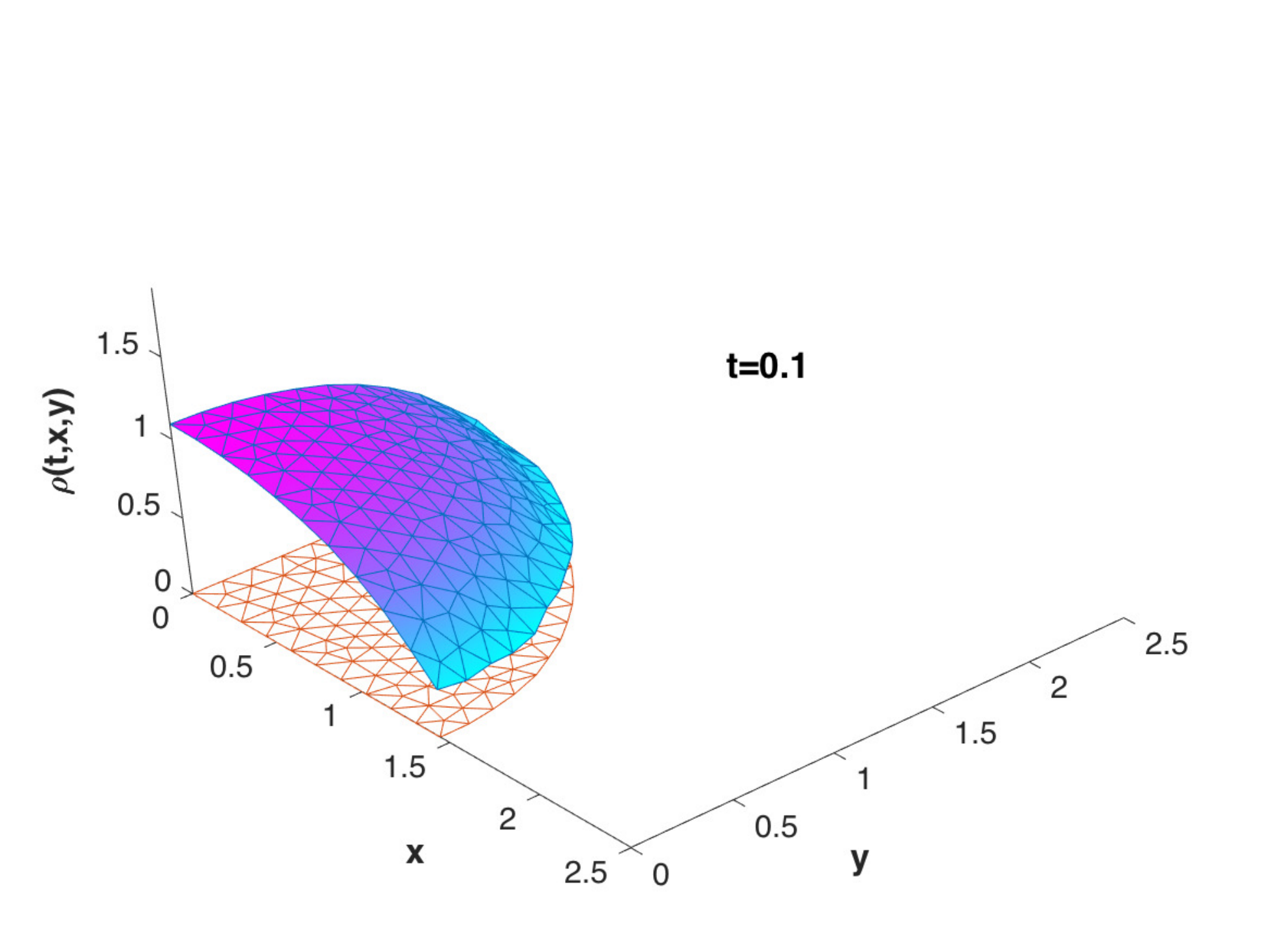}\\
  \includegraphics[width=0.45\textwidth]{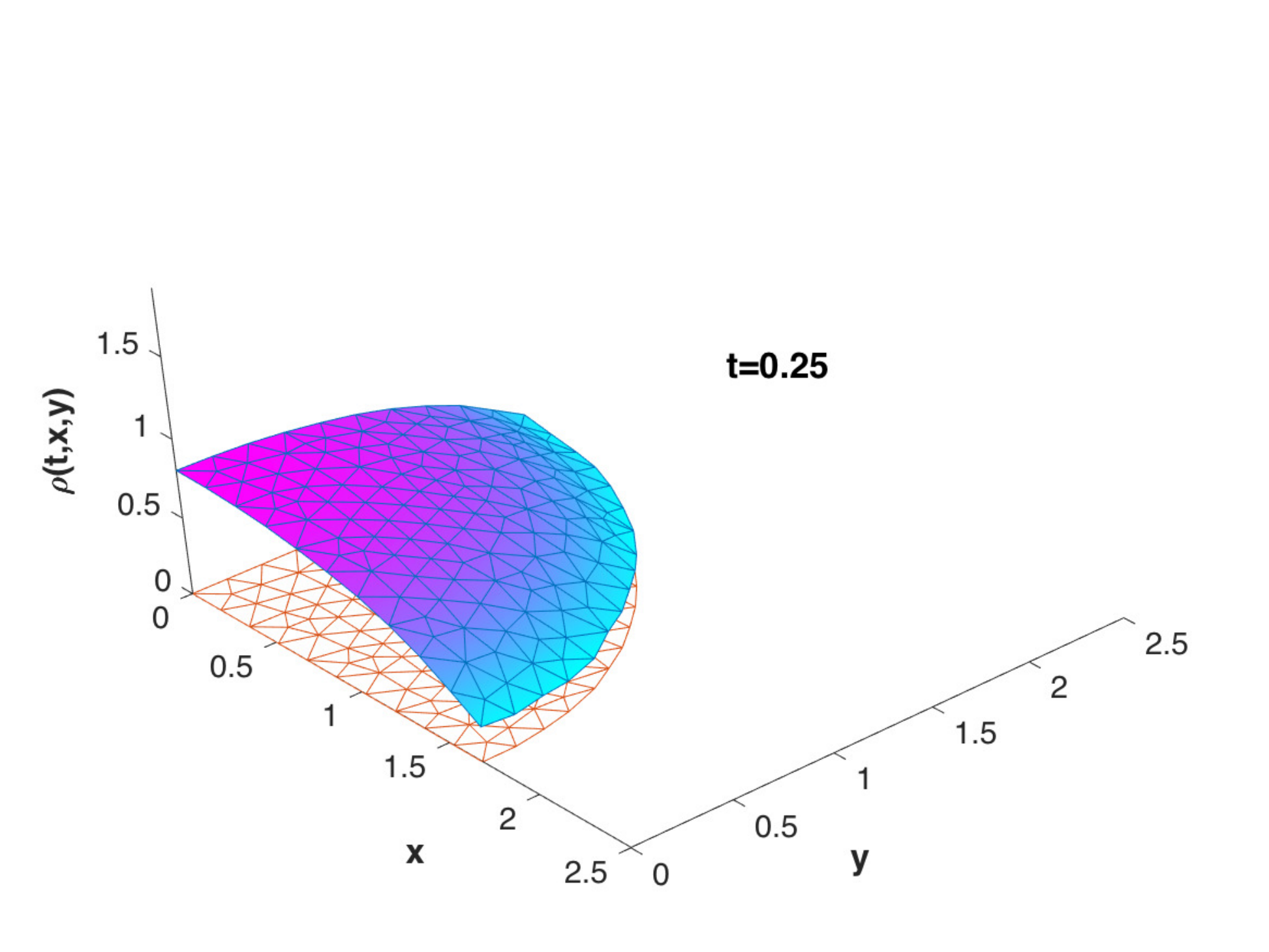}%
  \includegraphics[width=0.45\textwidth]{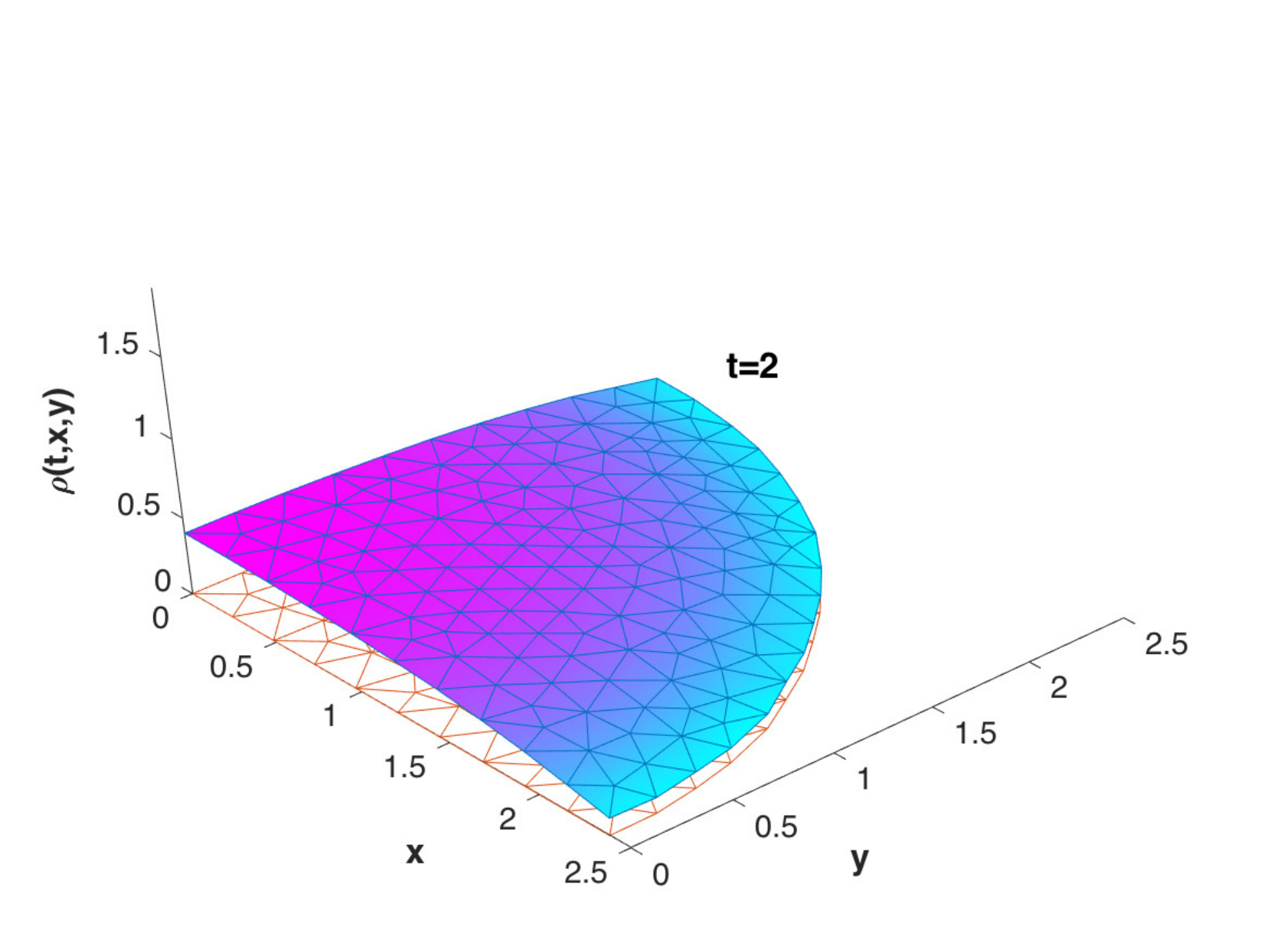}
  \caption{Numerical experiment 1: fully discrete evolution of our approximation for the self-similar solution
    to the free porous medium equation.
    Snapshots are taken at times $t=0.02$, $t=0.1$, $t=0.25$, and $t=2.0$.
  }
\label{fig:PMEevol}
\end{figure}

For illustration, we present some numerical experiments from \cite{gf-CDMM} for~\eqref{fpnonlinear}
with a cubic porous-medium nonlinearity $\Phi(r)=r^3$ and $V=0$. It is well-known (see, e.g., 
Vazquez~\cite{book:Vazquez}) that in the long-time limit $t\to\infty$, arbitrary solutions approach a 
self-similar solution called the Barenblatt profile. To reduce numerical effort, we imposed a four-fold symmetry 
of the approximation: we use the quarter circle as computational domain $K$, and interprete the discrete function 
thereon as one of four symmetric pieces of the full discrete solution. To preserve reflection symmetry over time,
homogeneous Neumann conditions are imposed on the artificial boundaries. This is implemented by reducing 
the degrees of freedom of the nodes along the $x$- and $y$-axes to tangential motion.

We initialize our simulation with a piecewise constant approximation of the Barenblatt profile at time $t=0.01$.
We choose a time step $\tau=0.001$ and the final time $T=2$.
In Figure~\ref{fig:PMEevol}, we have collected snapshots of the approximated density at different instances of time showing that 
the Barenblatt profile is well approximated. We refer for more details and other numerical tests to \cite[Section 6]{gf-CDMM}.

\subsection{Iteration of Lagrangian maps}
We shall now review the ansatz made in \cite{gf-JMO}.
The basic idea is to put no a priori restrictions on $X:\Theta\to\Omega$ itself, but on its iterative updates.
That is, we choose a finite dimensional set $\vansatz$ of vector fields $\velo:\Omega\to\setR^2$
that are tangential to the boundary of $\Omega$,
and ``discretize'' the variational problems \eqref{gf:mmL} in space
by minimizing only over Lagrangian maps of the form $X=(\id+\dt\velo)\circ X_\dt^{n-1}$, with $\velo\in\vansatz$.
The two motivations for this ansatz are the following:
\begin{itemize}
\item If $\vansatz$ only contains gradient vector fields,
  then the minimization problem \eqref{gf:mmL} inherits the convexity of \eqref{gf:mm}.
  Indeed, for $\velo=\nabla\varphi$, one has
  \begin{align*}
    \|(\id+\dt\velo)\circ X_\dt^{n-1}-X_\dt\|_{L^2_\theta} &= \|(id+\dt\velo)-\id\|_{L^2_{(X_\dt^{n-1})\#\theta}}\\
    &= \dt\wass\big((\id+\dt\velo)\#(X_\dt^{n-1})\#\theta,(X_\dt^{n-1})\#\theta\big),
  \end{align*}
  which preserves the equivalence between \eqref{gf:mmL} and \eqref{gf:mm}.
\item By choosing very regular ansatz functions in $\vansatz$,
  one expects to otbain a higher order of consistency in space,
  at least for approximation of smooth solutions.
\end{itemize}
The particular choice made in \cite{gf-JMO} are trigonometric polynomials,
\begin{align*}
  \vansatz
  = \left\{ \velo=\nabla\varphi\,\middle|\,
  \varphi(x) = \sum_{k_1,k_2=0}^Ka_{k,\ell}\cos(\pi k_1x_1)\cos(\pi k_2x_2)\right\}.
\end{align*}
The difficulty in the implementation of this scheme is that the Lagrangian map $X_\dt^n$ after the $n$th iteration
is the concatenation of $n$ maps of the form $\id+\dt\velo$, i.e.,
\begin{align*}
  X_\tau^n = (\id+\dt\velo_\dt^n)\circ\cdots\circ(\id+\dt\velo_\dt^1)\circ X_\tau^0.  
\end{align*}
Naturally, these concatenations cannot be evaluated explicitly,
and some further approximation needs to be performed for numerical approximation of the quantities in \eqref{gf:mmL}.

The ansatz made in \cite{gf-JMO} is to distribute ``test particles'' $\xi_m$ in $\Omega$
and monitor their motion and evolution of density under the Lagragian map;
these positions will be used as quadrature points for the integration in \eqref{gf:mmL}.
Specifically, after the $n$ step,
the position $x_m^n$ of $\xi_m$ and the density $\rho_m^n$ at $x_m^n$ are given by
\begin{align*}
  x_m^n:=X_\dt^n(\xi_m),
  \quad
  \rho_m^n:=\frac{\rho_m^0}{\det\dff X_\dt^n(\xi_m)}.
\end{align*}
Both quantities are easily obtained by iteration,
\begin{align*}
  x_m^n = x_m^{n-1}+\dt\velo^n(x_m^{n-1}),
  \quad
  \rho_m^n = \frac{\rho_m^{n-1}}{\det\big(\eins+\dt\dff\velo_\dt^{n}(x_m^{n-1})\big)}
\end{align*}
Finally, one assigns weights $\omega_m$ to the test particles,
that can be understood as particle masses. The choice
$\omega=1/M$ is obvious, however others improve the approximation quality of the integrals.
In the numerical experiments in \cite{gf-JMO},
we have partitioned the domain $\Omega$ into $K\times K$ square cells,
and have then used introduced equally distributed test particles with Lagrangian weights in each cell.

The iteration \eqref{gf:mmL} now amounts to minimizing
\begin{align}
  \label{gf:003}
  \sum_m\omega_k\left[\frac\dt2|\velo(x^{n-1}_k)|^2
  + h^\#\left(\frac{\rho_m^{n-1}}{\det\big(\eins+\dt\dff\velo(x_m^{n-1})}\right)
  + V\big(x_m^{n-1}+\dt\velo(x_m^{n-1})\big)\right].
\end{align}
The Euler Lagrange equations produce a a non-linear system of equations for the coefficients $a_{k,\ell}$ in $\velo_\dt^n$.

The main results of \cite{gf-JMO} are:
\begin{itemize}
\item The numerical scheme performs well in a variety of experiments.
\item The functional \eqref{gf:003} is uniformly convex of modulus $\frac1\dt+\lambda$ with respect to $\velo\in\vansatz$
  (and thus strictly convex in terms of the coefficients $a_{k,\ell}$ for $\dt>0$ small enough).
\item Assuming an a priori uniform bound on the density functions in $C^5(\Omega)$,
  one obtains convergence of the Lagrangian maps to a solution of \eqref{gf:1}.
\end{itemize}

\subsection{Lagrangian maps from Laguerre cells}
We briefly mention a semi-Lagrangian method developed in \cite{gf-BCMO}
for numerical approximation of solutions to the Fokker-Planck equation \eqref{gf:FP}.
The scheme consists of two steps, that are carried out alternatingly:
a Lagrangian one, in which a measure concentrated in $N$ points
is transformed into a piecewise constant density function,
and a projection step, in which the density function is again concentrated in $N$ point masses.

A central role is played by Laguerre cells, which we recall for convenience of the reader.
Given $N$ points $\{x_1,\ldots,x_N\}\in\Omega$ with respective real weights $\{\varphi_1,\ldots,\varphi_N\}\in\R$,
the corresponding $N$ Laguerre cells $\{L_1,\ldots,L_N\}\subset\Omega$ are defined as follows:
the $i$th cell $L_i$ consists of all $x\in\Omega$ such that $x\cdot x_i-\varphi_i\ge x\cdot x_j-\varphi_j$ for all $j$s.
In the special case that the weights are $\varphi_i=\frac12|x_i|^2$ for all $i$,
the $L_i$ are precisely the Voronoi cells for the point configuration $x_1,\ldots,x_N$.
For general weights, the geometry of Laguerre cells can be much more complicated;
for instance, one should not expect $x_i\in L_i$,
and one can even have $L_i=\emptyset$ for one or several $i$'s.
In any case, the $L_1,\ldots,L_N$ cover $\Omega$ completely,
and the intersection of two Laguerre cells is a convex set of lower dimension, possibly empty.
In fact, it is obvious from the definition that $L_i\cap L_j$ lies in a hyperplane orthogonal to $x_i-x_j$.
There is a characterization of Laguerre cells in terms of subdifferentials of convex functions.
Namely, let $\varphi:\Omega\to\R$ be the largest convex function
such that, on the one hand, $\varphi(x_i)=\varphi_i$ at every $i$,
and on the other hand, that $\nabla\varphi(\Omega)\subseteq\Omega$.
This $\varphi$ is globally Lipschitz, and has corner singularities at each $x_i$.
Recalling the definition of the subdifferential $\partial\varphi$,
it is easily seen that $\partial\varphi(x_i)=L_i$.

The ansatz for the Lagrangian step of the scheme in \cite{gf-BCMO} is to perform a minimizing movement,
i.e. to solve the minimization problem \eqref{gf:mm} with $\nrg=\ent_{h,0,0}$.
The ``old'' measure $\rho_{\Delta t}^{n-1}$ is assumed to consist of $N$ Diracs
at positions $\{x_1,\ldots,x_N\}\in\Omega$ with respective masses $\{m_1,\ldots,m_N\}$.
The set of admissible $\rho$ in the minimization of \eqref{gf:mm} consists of absolutely continuous densities
with the property that $\rho$ is piecewise constant on the $N$ Laguerre cells $\{L_1,\ldots,L_N\}\subset\Omega$
for the given positions $\{x_1,\ldots,x_N\}$ and some weights $\{\varphi_1,\ldots,\varphi_N\}$,
with the total mass in each $L_i$ equal to $m_i$.
That is, the weights $\{\varphi_1,\ldots,\varphi_N\}$ parametrize the space of admissible $\rho$'s.
The subsequent projection step consists in determining some point in each $L_i$, e.g. the Steiner point,
and concentrating the mass $m_i$ there.

There are several reasons for this choice of the Lagrangian step.
One is that the Wasserstein distance from each admissible $\rho$ to the datum $\rho_{\Delta t}^{n-1}$ can be calculated explicitly.
In fact, the aforementioned convex function $\varphi$ which extends the values of $\varphi_i$ at the $x_i$
is a Kantorovich potential for this transport.
Another reason, which also makes this scheme very special,
is that geodesic convexity of $\ent_{h,0,0}$ is inherited under the discretization.
More precisely, the map from the $N$-vector of values $\{\varphi_1,\ldots,\varphi_N\}$
to the functional value $\ent_{h,0,0}(\rho)$ at the respective piecewise constant density $\rho$ is convex.

In \cite{gf-BCMO}, the ability of the numerical scheme is verified in a variety of experiments.
On the analytical side, a rigorous proof is given for the $\Gamma$-convergence of the Lagrangian step
to one step in the corresponding minimizing movement scheme in the limit of spatial refinement.

\subsection{Particle approximation: blob method}

Another Lagrangian approach to nonlinear aggregation-diffusion equations of the form \eqref{gf:FP} has been pursued in \cite{carrillo2017blob}, also used in \cite{landau} for the Landau equation. The main strategy consists in using as approximation space for densities the set of finite linear combinations of Dirac Delta distributions. In other words, we want to approximate the gradient flow equations by a gradient flow in finite dimensions on the locations of these Dirac Deltas that we called particle locations. This approach is quite natural without diffusion $\Phi=0$. In fact, it is very much connected to the fact that the aggregation equation 
\begin{align}
  \label{gf:aggreg}
  \partial_t\rho_t = \dv\big(\rho_t\,[\nabla V + \rho_t\ast \nabla W]\big),
\end{align}
can be seen as the mean-field limit \cite{CarrilloChoiHauray,gf-CFetc,Jabin,CCHS} of a particle system of the form
 \begin{align} \label{partsol1} 
 	\dot{x}_i = -\nabla V(x_i) - \sum_{j \neq i } \nabla W(x_i - x_j) m_j .
\end{align}
This can be formally understood by taking the ansatz
$$
\rho_t \approx \rho^N_t =  \sum_{i = 1}^N \delta_{x_i(t)} m_i, 
$$
where $\delta_{x_i}$ is a Dirac mass centered at $x_i \in \Omega$, into \eqref{gf:aggreg} as a distributional solution. The initialization is done by discretizing the initial datum $\rho_0$ as a finite sum of $N$ Dirac masses,
\begin{align} \label{partinitial1}
 \rho_0 \approx \rho_0^N = \sum_{i = 1}^N \delta_{x_i^0} m_i, \qquad x_i^0 \in \Omega, \quad m_i \geq 0 .
 \end{align}
The particle method (\ref{partsol1})  provides a semi-discrete numerical method preserving the gradient flow structure since the discrete energy $\ent_{0,V,W}(\rho^N_t)$ is decreased along the solutions of  (\ref{partsol1}). In fact, the system (\ref{partsol1}) is a finite dimensional gradient flow of the discrete interaction energy $\ent_{0,V,W}(\rho^N_t)$ seen as a function of the particle locations. 

The approach to generalize these deterministic particle methods for diffusive equations is not that obvious since we cannot evaluate the energy $\ent_{h,0,0}(\rho)$ for finite linear combinations of Dirac Delta distributions. The main novelty introduced in \cite{carrillo2017blob} is to regularize the entropy functionals associated to diffusion equations. More precisely, let us take a mollifier $\varphi_\epsilon(x) = \varphi(x/\epsilon)/\epsilon^d$, $\epsilon >0$ for a given $\varphi$ smooth positive fast decaying function integrating to unity, i.e., we choose an approximation of the Dirac Delta at the origin. To showcase the regularization we propose, let us consider the particular case of the porous medium equation. Then, we propose to approximate the energy functional by 
 \begin{align*} 
 \ent_m(\rho)=\int_\Omega \rho^{m}\,dx \approx
 \ent_m^\epsilon(\rho)=\int_\Omega (\varphi_\epsilon*\rho)^{m-1}\rho \,dx.
\end{align*}
The advantage of the approximated energy $\ent_m^\epsilon(\rho)$ is that it makes sense for finite linear combination of Dirac Deltas. 
Other regularizations of the functional are possible, we chose this one since it leads to a pure particle system avoiding any continuous convolution as seen below.

The ultimate objective is then to approximate the gradient flow of the energy $\ent_m(\rho)$ by the gradient flow of the regularized energy $\ent_m^\epsilon(\rho)$. Finally, the regularized gradient flow can be approximated by the particle method in the same spirit as done earlier for the pure aggregation equation \eqref{gf:aggreg}. Note that the case $m=2$ is special since we approximate the porous medium equation with $m=2$ by the aggregation equation with a repulsive smooth potential.

In general, giving the energy functional $\ent_{h,V,W}(\rho)$ and assuming that $\Phi(r)=rF(r)$, we define the regularized functional as
\begin{align} \label{regularizedentropygen} 
	{\mathcal F}_\epsilon(\rho) =  \int_\Omega F(\varphi_\epsilon*\rho)\rho\, dx ,
\end{align}
and 
\begin{align} \label{regularizedentropygen2} 
	{\mathcal E}_\epsilon(\rho) =  \int_\Omega F(\varphi_\epsilon*\rho)\rho\, dx +\int_\Omega \big[\rho V + \frac12 \rho( W\ast\rho)\big]\dd x.
\end{align}
The gradient flow associated to this functional leads to $\partial_t\rho_t = \dv\big(\rho_t\nabla\frac{\delta{\mathcal E}_\epsilon}{\delta\rho}\big)$ with the variations given at any density $\rho$ of ${\mathcal F}_\epsilon$ given by
$$
	\frac{\delta \mathcal{F}_\epsilon}{\delta \rho} (\rho)= \varphi_\epsilon * (F' \circ (\varphi_\epsilon*\rho)\rho) + F\circ(\varphi_\epsilon *\rho).
$$
Then $\rho^N_t$ satisfies the weak formulation $\partial_t\rho_t = \dv\big(\rho_t\nabla\frac{\delta{\mathcal F}_\epsilon}{\delta\rho}\big)$ if and only if the particles follow the system
\begin{align} \label{ODEsystem}
	\begin{cases}
\displaystyle \dot{x}_i(t) =-\nabla V(x_i(t))- \sum_{j \neq i} \nabla W(x_i(t)-x_j(t))m_j- \nabla \frac{\delta {\mathcal F}_\epsilon}{\delta \rho} \left(\sum_j \delta_{x_j(t)} m_j \right) (x_i(t)), \\
		 x_i(0) = x_i^0,
	\end{cases}
\end{align}
for $t\in[0,T]$. It is proven in \cite[Corollary 5.5]{carrillo2017blob} that the ODE system \eqref{ODEsystem} is well posed under suitable convexity and growth conditions of the potentials in case of working in the whole space for the power nonlinear diffusion $F(r)=r^m$ and $m\geq 2$. Moreover, the particle method was shown to be convergent to the solutions of \eqref{gf:FP}
under suitable conditions, see \cite[Theorem 5.6]{carrillo2017blob}. Notice that the right hand side of \eqref{ODEsystem}  can be expanded as
$$
\nabla \frac{\delta {\mathcal F}_\epsilon}{\delta \rho} = 
\nabla \varphi_\epsilon * (F'  \circ \varphi_\epsilon*\rho)\rho) + F'\circ(\varphi_\epsilon *\rho) (\nabla \varphi_\epsilon *\rho),
$$
and that its evaluation in $\rho_t^N$ leads to the elimination of all convolutions by finite sums. In particular for linear diffusion leads to
$$
\nabla \frac{\delta {\mathcal F}_\epsilon}{\delta \rho} = 
\nabla \varphi_\epsilon * \left(\frac{\rho}{\varphi_\epsilon*\rho}\right) + \frac{\nabla \varphi_\epsilon *\rho}{\varphi_\epsilon *\rho}.
$$
Let us remark that the system \eqref{ODEsystem} is a finite dimensional gradient flow in $\R^{dN}$ of the discrete regularized entropy functional ${\mathcal E}_\epsilon(\rho^N)$ seen as a function of the particle positions  $\{x_1,\ldots,x_N\}\in\R^d$. Therefore, the system 
\eqref{ODEsystem} keeps the gradient flow structure at the semidiscrete level and the discrete regularized energy is dissipated according to the same law of the continuous problem given by
$$
\frac{d}{dt} {\mathcal E}_\epsilon(\rho^N(t)) = - \int_\Omega \left|\nabla \frac{\delta {\mathcal F}_\epsilon}{\delta \rho}(\rho^N(t))\right|^2
\rho^N(t) \, dx.
$$ 
\begin{center}
\begin{figure}[ht!] 
\centering
\includegraphics[scale=.7]{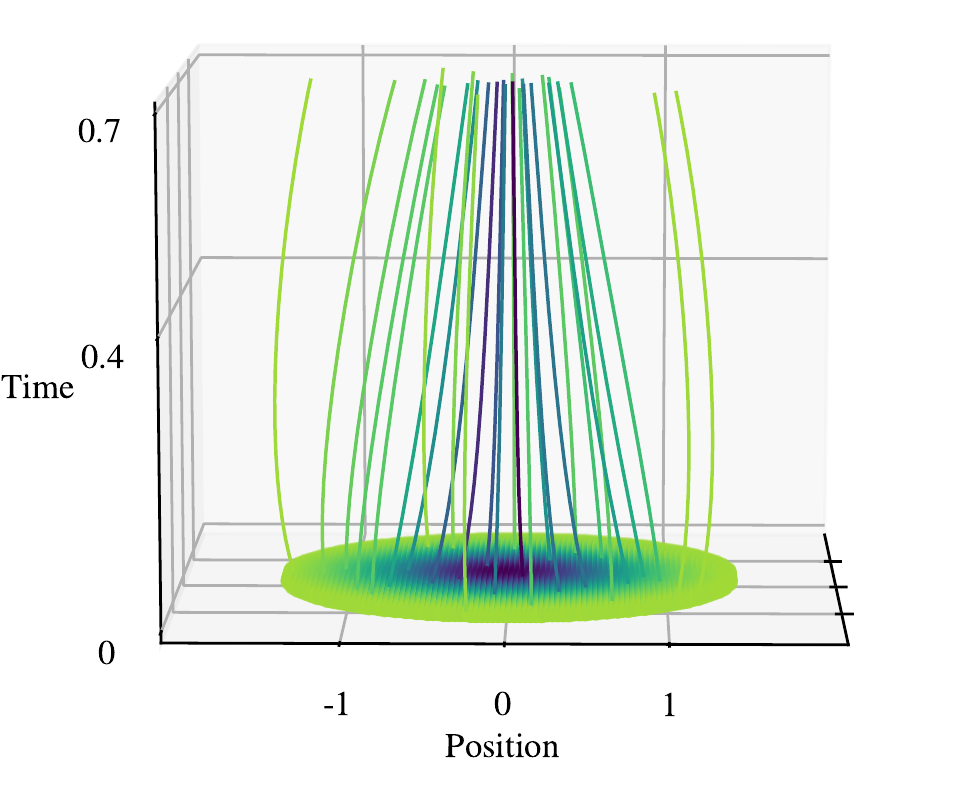}
\caption{Two-Dimensional Keller--Segel Equation: Blowup with Supercritical Mass $9 \pi$. Evolution of particle trajectories, colored according to the relative mass of each trajectory.} \label{2DKSSup}
\end{figure}
\end{center}
We illustrate the particle method  \eqref{ODEsystem} with some examples extracted from \cite[Section 6]{carrillo2017blob} in which the mollifier is chosen as a Gaussian.
We consider the classical Keller--Segel equation ($V=0$, $W(x) = 1/(2 \pi)\log|x|$, $m=1$) in two dimensions without normalizing the mass of the density. There is a dichotomy 
between global existence and blow-up given by the critical mass $8 \pi$, and in particular, for supercritical initial data, solutions blow up in finite time \cite{DP,BDP}.
In Figure \ref{2DKSSup}, we show the particles for the case of supercritical mass $9\pi$. Indeed, one of the benefits of our blob method approach is that the numerical method naturally extends to two and more dimensions, and we observe similar numerical performance independent of the dimension. We also plot the evolution of particle trajectories, observing the 
tendency of trajectories in regions of larger mass to be driven largely by pairwise attraction, while trajectories in regions of lower mass feel more strongly the effects of diffusion.

In Figures \ref{2DKSden} and \ref{2DKSSup}, the initial data is given by a Gaussian mollifier scaled to have mass that is either supercritical ($> 8 \pi$), 
critical ($=8 \pi$), or subcritical ($< 8 \pi$) with respect to blowup behavior \cite{DP,BDP}. In Figure \ref{2DKSSup}, we observe how the particles associated to 
initial data with supercritical mass aggregate at the origin.
\begin{center}
\begin{figure}[ht]
\centering
\textbf{Two-Dimensional Keller--Segel Equation: Evolution of Density}
\vspace{0cm}
\begin{flushleft} \hspace{1.15in} t = 0.0 \hspace{1.15in} 
 t = 0.15 \hspace{.1in} \end{flushleft}
\vspace{-.1cm}

{\footnotesize Subcritical Mass $=7\pi$}

\hspace*{-1.2cm}
\includegraphics[trim={.83cm 1cm 1.6cm 1cm},clip,scale=.85]{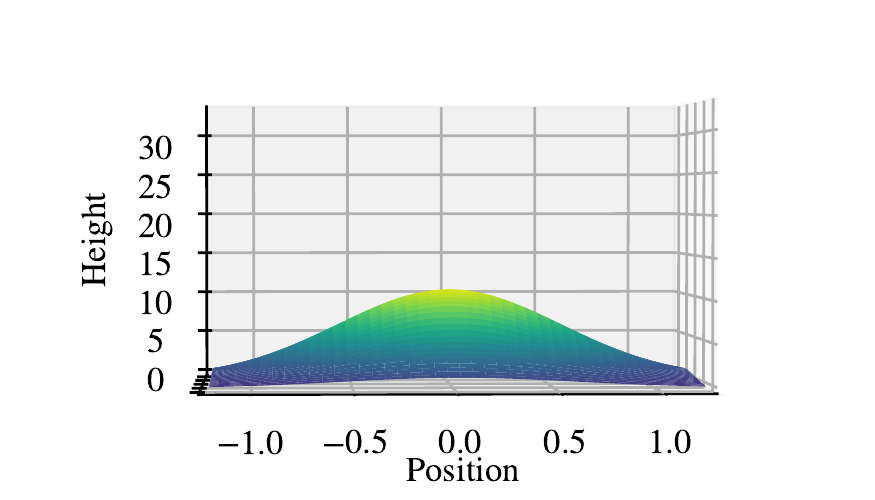}
\includegraphics[trim={2cm 1cm 1.6cm 1cm},clip,scale=.85]{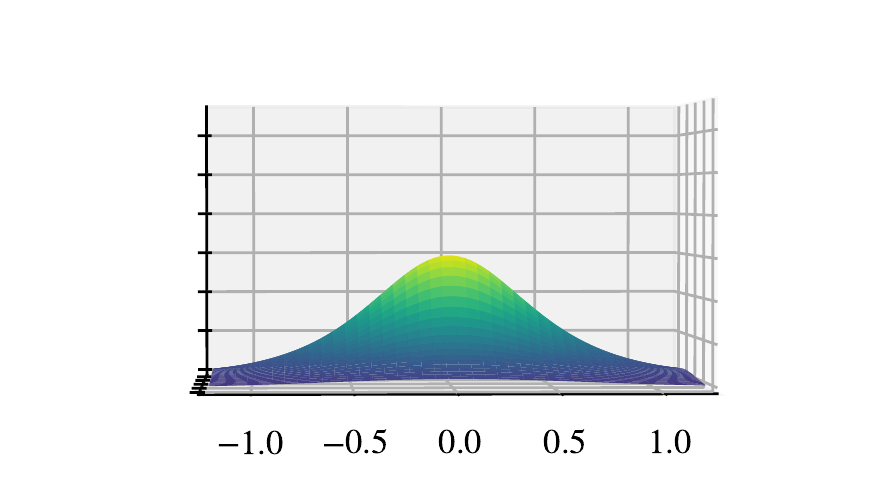}

\vspace{-.1cm}
{\footnotesize Critical Mass $=8\pi$}

\includegraphics[trim={2.03cm 1cm 1.6cm 1cm},clip,scale=.85]{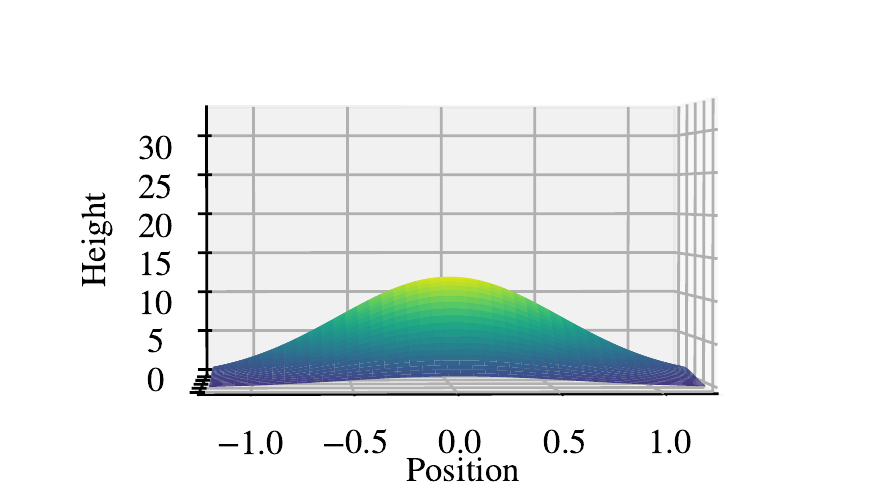}
\includegraphics[trim={2cm 1cm 1.6cm 1cm},clip,scale=.85]{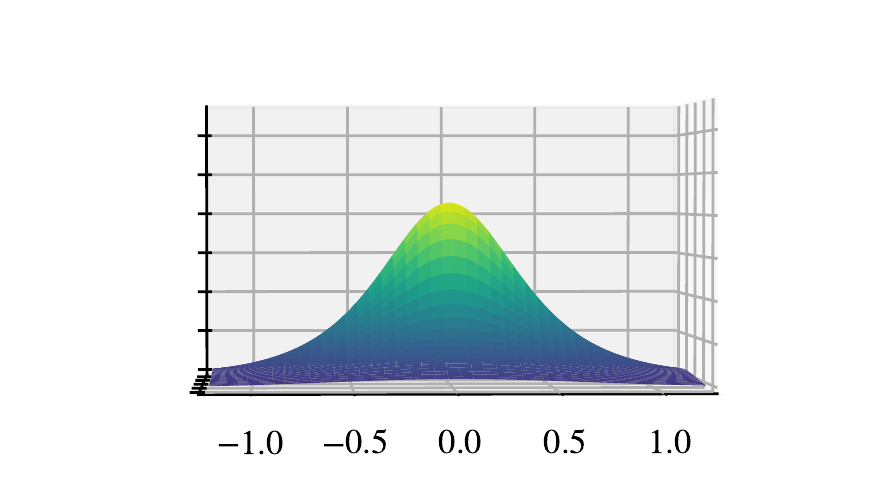}

\vspace{-.1cm}
{\footnotesize Supercritical Mass $=9\pi$}

\includegraphics[trim={2.03cm .1cm 1.6cm 1cm},clip,scale=.85]{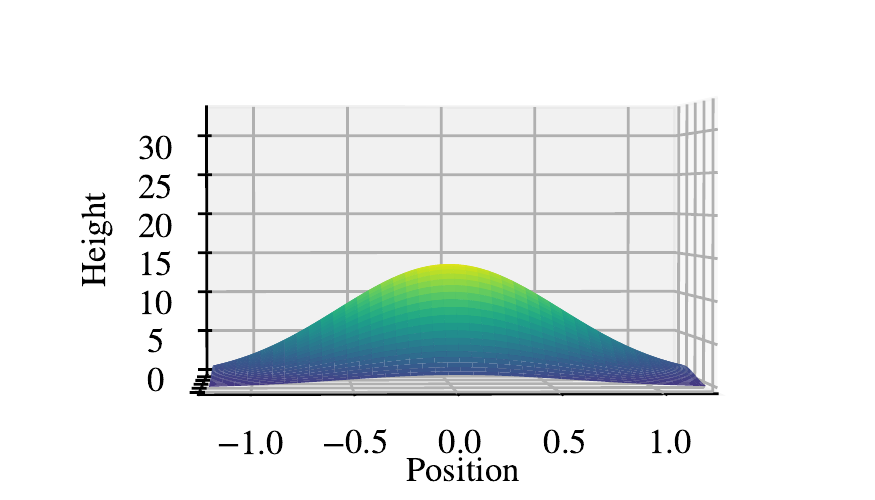}
\includegraphics[trim={2cm .1cm 1.6cm 1cm},clip,scale=.85]{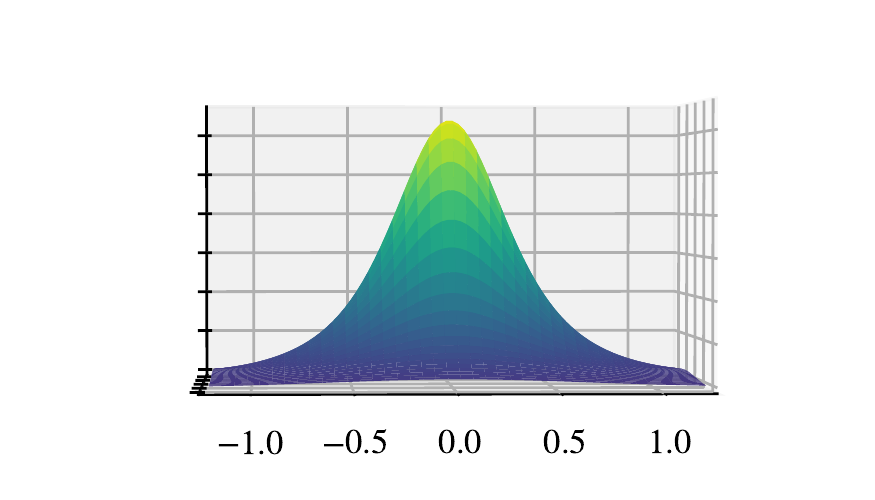}
\vspace{-.2cm}
\caption{Evolution of numerical solutions for the two-dimensional Keller--Segel equation with subcritical, critical, and supercritical initial data. } \label{2DKSden}
\end{figure}
\end{center}

\subsection{Particle approximation: Yosida regularization}
We briefly comment on a recent article \cite{gf-LMSS},
where the same idea as in \cite{carrillo2017blob} is pursued,
but with a different regularization procedure.
Again, a particle scheme for solution of the Fokker-Planck equation \eqref{gf:FP} is constructed,
such that the empirical measure $\rho^N_t$ associated to $N$ moving particles in $\R^d$
approximates the the solution $\rho_t$ in the weak-$\star$ sense for measures. 
Here the $N$ particle positions $\{x_1,\ldots,x_N\}\in\R^d$ obey an ODE system determined by the gradient flow on $\R^{dN}$ of a functional $\ent^\epsilon$, similarly to (but less explicit than in) equations \eqref{ODEsystem}.
We remark that \cite{gf-LMSS} focuses on the \emph{linear} Fokker-Planck equation with $W\equiv 0$
as well as on a model for crowd motion,
but the basic idea directly generalizes to non-linear Fokker-Planck equations.

As discussed above, one cannot directly evaluate the energy $\ent_{h,V,W}$ on the empirical measure $\rho^N$.
Instead of ``smearing out'' the particles with a mollifier as in \eqref{regularizedentropygen},
the Yosida approximation of the internal energy is used, i.e.,
\begin{align*}
  \ent_{h,V,W}^\epsilon(\rho) = \inf_\sigma\left[\frac1{2\epsilon}\wass(\rho,\sigma)^2 + \int_\Omega h(\sigma)\,dx\right]
  + \int_\Omega\big[\rho V+\frac12\rho(W\ast\rho)\big]\,dx,
\end{align*}
For $\rho=\rho^N$ the empirical measure of $N$ particles, the second integral reduces to a finite sum.
The Yosida approximation cannot be calculated explicitly (which is in contrast to the procedure in \cite{carrillo2017blob} explained above),
but efficient methods for its numerical approximation have been developed recently
in the context of semi-discrete optimal transport \cite{gf-KMT}.
The core idea is to obtain the Kantorovich potential $\psi$,
which only needs to be defined at the $N$ particle positions $x_1,\ldots,x_N$,
by adjusting the corresponding Laguerre cells such that they satisfy a certain geometric condition.

Unconditional convergence of the empirical measures $\rho^N_t$ to the solution $\rho_t$ of \eqref{gf:FP}
in the joint limit $N\to\infty$ and $\epsilon\to0$ has been proven in dimension $d=1$.
In higher space dimension, convergence can be shown if a certain uniform a priori bound is satisfied.
Convexity of the regularized functional is not discussed;
instead, it is shown that $\ent_{h,V,W}^\epsilon(\rho^N)$ is always semi-\emph{concave} as a function on $\R^{Nd}$,
which implies differentiability at all particle configurations with $N$ distinct point.
After discretization in time by the implicit Euler method, a fully practical numerical scheme is obtained.

\section{Optimize then discretize}\label{s:otd}
\noindent We conclude by adopting the reverse strategy, that is first optimize then discretize. 
This idea was first proposed in \cite{CM2009} and further developed in \cite{CRW2016}. 
Both methods are based on different temporal and spatial discretizations of the 
optimality system  to the respective $L^2$ gradient flow structure. This approach is particularly interesting 
in higher space dimension, where the calculation of the Wasserstein distance is computationally too expensive.
The optimality system of \eqref{gf:FPehash} corresponds to a highly non-linear parabolic PDE for the diffeomorphism $X$.
We will focus on the most general computational approach in the following; that is an implicit in time and finite element discretization in 2D, 
since all other schemes correspond to straight-forward simplifications.\\

\noindent We start from the optimality system of the $L^2$ gradient flow \eqref{gf:mmL} with the transformed relative entropy \eqref{gf:FPehash}. 
Then the diffeomorphism $X$ satisfies the following nonlinear PDE system in the limit $\Delta t \rightarrow 0$:
\begin{align}\label{e:diffeo}
\partial_t X &=  \dv\left[ \Psi'(\det DX) (\cof DX)^T \right] - \nabla V \circ X - \int_{\Theta} \nabla W(X(\xi)- X(\xi')) d\xi',
\end{align}
We recall that the diffeomorphism $X$ corresponds to the map from a general reference measure to the unknown probability density $\rho_t$. 
We will consider the simplest case - that is the uniform density as a reference measure on $[0,1]^2$ - in the following. Then $\rho_t$ can be computed for
sufficiently smooth $X$ via \eqref{gf:X2rho}:
\begin{align}\label{e:varchange}
\rho(X(\xi)) \det(DX(\xi)) = 1.
\end{align}

\subsection{Numerical scheme: multiD and finite element approach}
\noindent We start by stating the semi-discrete implicit discretisation, where $\Delta t$ denotes the discrete time step, $t^{n+1} = (n+1) \Delta t$ and that $X^{n+1}$ corresponds to the solution 
$X = X(\xi,t)$ at time $t^{n+1}$. Then \eqref{e:diffeo} reads as
\begin{align}\label{e:implicitdiffeo}
\begin{split}
\frac{X^{n+1}-X^n}{\Delta t} = \,& \dv [\Psi'(\det DX^{n+1}) (\cof DX^{n+1})] \\
&-   \nabla V(X^{n+1}) - \int_{\Theta } \nabla W(X^{n+1}(\xi)-X^{n+1}(\xi')) d\xi.
\end{split}
\end{align}
Its variational formulation for test functions $\varphi = \varphi(\xi) \in H^1(\Theta)$ defines a  nonlinear operator $F$, which is given by:
\begin{align}\label{e:F}
\begin{split}
F(X,&\varphi) = \frac{1}{\Delta t} \int_{\Theta}(X^{n+1}-X^n) \varphi(\xi) d\xi + \int_{\Theta} \Psi'(\det DX^{n+1}) (\cof DX^{n+1}) \nabla \varphi(\xi) d\xi \\
+ &\int_{\Theta} \nabla V(X^{n+1}) \varphi(\xi) d\xi +\int_{\Theta} \left[\int_{\Theta} \nabla W(X^{n+1}(\xi)-X^{n+1}(\xi')) d\xi' \right]  \varphi(\xi) d\xi .
\end{split}
\end{align}
Then the fully discrete formulation can be obtained by choosing a suitable spatial discretisation. We choose lowest order $H^1$ conforming finite elements, that is 
elementwise linear functions, for $X$ as in \eqref{e:fem}. Note that we will use the same notation for the infinite dimensional and finite dimensional testfunctions to 
enhance reability in the following. The spatial discretisation defines the nonlinear operator equation $F(X,\varphi)=0$, which can be solved using Newton-Raphson's method.  
In doing so, we compute the Jacobian matrix $DF$ of \eqref{e:F} as well as the $(k+1)$-th Newton update $Y^{n+1,k+1}$ via
\begin{align}\label{e:newtonupdate}
DF(X^{n+1,k}, \varphi) Y^{n+1,k+1} =-F(X^{n+1,k},\varphi),
\end{align}
for all test functions $\varphi(\xi) \in H^1(\Theta)$. 
Note that the Jacobian matrix $DF$ is a full matrix and has no sparse structure due to the convolution operator $W$. \\
This discretisation can be used for a very general class of equations. Solutions to these equations often exhibit complex features, such as compact supports or concentration phenomena. 
This corresponds to diffeomorphisms $X_t$ becoming degenerate, and it is therefore often useful to perform a damped Newton update via:
\begin{align*}
X^{n+1,k+1} = X^{n+1,k} + \alpha Y^{n+1,k+1},
\end{align*} 
where $0 < \alpha < 1$ is a suitably chosen damping parameter. The Newton iteration \eqref{e:newtonupdate} is terminated when a stopping criterion 
\begin{align*}
  \lvert F(X^{n+1,k+1},\varphi) \rvert \leq \epsilon_1 \text{ or } \lVert X^{n+1,k+1}-X^{n+1,k} \rVert \leq \epsilon_2,
\end{align*} 
for given error bounds $\epsilon_1>0$ and $\epsilon_2>0$ is satisfied. While the implicit in time discretization involves the solution of a nonlinear PDE system, 
it does not impose any CFL type condition on the time step as in the explicit case, see \cite{CM2009}.\\

\noindent The 'optimise-than-discretise' approach requires the transformation of the no-flux boundary conditions \eqref{e:bc} as well as the computation of 
the initial diffeomorphism $X_0$ given $\rho_0$.
The is not straight forward in higher space dimension and we will discuss the main ideas in $2D$ in the following.

 \noindent\subsubsection*{Boundary conditions: } 
 To formulate the respective no-flux boundary conditions for $X_t$  we consider diffeomorphisms, which map the boundary of the reference domain 
 $\partial \Theta$ onto $\partial \Omega$ without rotations only.
Then the translated no-flux boundary conditions \eqref{e:bc} are given by
\begin{align}\label{e:bcdiffeo}
  \nml^T (\cof DX)^T \partial_t X = (\cof DX) \nml \cdot \partial_t X  = 0.
\end{align} 
\noindent Note that \eqref{e:bcdiffeo} implies different conditions for different computational domains. Consider for example a rectangular mesh on 
$\Theta = [0,1]^2$ and $\Omega = [0,1]^2$. Then careful calculations yield
\begin{align*}
\partial_{\xi_1} X_2 = 0 \text{ for } \xi_1 = 0, \xi_1 = 1 \text{ and } \partial_{\xi_2} X_1 = 0 \text{ for } \xi_2 = 0, \xi_2 = 1.
\end{align*}
In case of a circles of radius $R$ equation \eqref{e:bcdiffeo} translates to
\begin{align}\label{e:bc2}
\sin \theta \partial_t X_2 \partial_{\xi_1} X_1 + \cos \theta \partial_t X_1 \partial_{\xi_2} X_2 &= 0,
\end{align}
 which implies 
\begin{align}
X_1(\xi,t) = X_2(\xi,t) = \text{Id}.
\end{align}
This yields $\partial_t X_1 =\partial_t X_2 = 0$ on the boundary.  


\subsubsection*{Pre-processing and post-processing: calculating the initial diffeomorphism  and the final density}

We start by discussing how the initial diffeomorphism $X_0$ can be calculated given an initial density $\rho_0$. Different approaches have been proposed in the 
literature - depending on the spatial discretization of the underlying domain. We will review two possible constructions -  the first is based on a splitting approach 
in case of rectangular meshes, while the second uses density equalising maps for triangular meshes.

\paragraph*{Rectangular mesh: } This approach is based on splitting the problem in the $\xi_1$ and $\xi_2$ direction and solving the respective Monge-Kantorovich problems
in each direction. This gives the following one-dimensional Monge-Kantorovich problem in the 
$\xi_1$ direction: determine $a_i$ at every mesh point $\xi_{1,i} \in [0,1)$ such that
\begin{align*}
 \int_0^{a_i} \int_0^1 \rho_0(\eta, \zeta)d\zeta d\eta = \xi_{1,i}.
\end{align*}
Next we solve the Monge-Kantorovich problem in the $\xi_2$ direction. Hence we have to find $b_{ij}$ (which corresponds 
to the discrete value of a function $b$ at a grid point $\xi_{ij} = (i\Delta x, j \Delta y)$) such that
\begin{align*}
 \int_0^{b_{ij}}\rho_0(a_i,\eta)d\eta = \frac{\xi_2}{\sqrt{M}} \int_0^1 \rho_0(a_i, \eta)d\eta.
\end{align*}
The initial diffeomorphism is then given by $X_0(\xi_{ij}) = (a_i, b_{ij})$.
\paragraph*{Quadrilateral or triangular mesh:} In the case of general quadrilateral and triangular meshes the above construction does not work. Then one can solve 
the corresponding Monge Ampere  equation (giving the optimal transportation plan in case of quadratic cost), use Knote theory or use 
  density equalizing maps instead. We shall outline the latter approach, which is based on \cite{M1965}, further studied in \cite{ASMV2003} and also used in cartography \cite{GN2004}. 
  Hereby  the initial diffeomorphism is constructed by following the heat flow, which transports an initial density to the uniform density, backwards in time. 
This approach is very flexible, since it can be used for general domains and only requires a fast heat equation solver.\\
Consider the heat equation on a bounded domain $\Omega \subset \mathbb{R}^2$ and solve
\begin{align}\label{e:cons}
\partial_t \rho + \Div(\rho w) = 0, \qquad \mbox{with } w  = -\frac{\nabla \rho}{\rho},
\end{align}
with initial datum $\rho(0,x) = \rho_0(x)$ and homogeneous Neumann boundary conditions. The heat equation then transports the initial datum 
$\rho_0$ via the velocity field $w  = -\frac{\nabla \rho}{\rho}$ towards its equilibration density 
$\bar{w} = \frac{1}{\lvert \Omega \rvert} \int_{\Omega} \rho_0(\xi) d\xi$ as $t \rightarrow \infty$. 
Then the cumulative displacement $\mathbf{x}(t)$ of any point at time $t$ is determined by integrating the velocity field, which corresponds to solving
\begin{equation*}
\mathbf{x}(t)=\mathbf{x}(0)+\int_0^t w(t',\mathbf{x}(t'))\,dt'.
\end{equation*}
As $t\rightarrow \infty$, the set of such displacements for all points $x = x(t)$ in $\Omega$, that is the grid points of the computational mesh, defines 
the new density-equalized domain. Note that we actually have to compute its inverse, since we need to find  the map which maps the constant density to the initial density $\rho_0$. 
\subsubsection*{Post-processing:} To obtain the final density $\rho_T := \rho(\xi,t=T)$ from the final diffeomorphism $X_T := X(\xi,t=T)$ we solve a 
regularized version of \eqref{e:varchange} (in 2D) given by:
\begin{align}\label{e:regvarchange}
\varepsilon \Delta \rho_T(X_T(\xi)) + \rho_T(X_T(\xi)) = \frac{1}{\det DX_T(\xi)} \text{ with } 0 < \varepsilon \ll 1.
\end{align}
This can be accomplised by solving for example the respective variational formulation using finite elements. Note that the regularisation $\varepsilon \Delta \rho$
ensures a stable reconstruction in case of compactly supported or aggregated solutions. 

\subsubsection*{Simulations}
We conclude by illustrating the dynamics with different examples in 1D and 2D. All results are based on the implicit in time discretization \eqref{e:implicitdiffeo}, the spatial
discretisation uses finite differences 1D and finite elements in 2D.
\begin{figure}[h!]
\begin{center}
\subfigure[Reconstructed BP profile  $\rho = \rho(x,t)$ at time $t=0.021$.]{\includegraphics[width=0.42\textwidth]{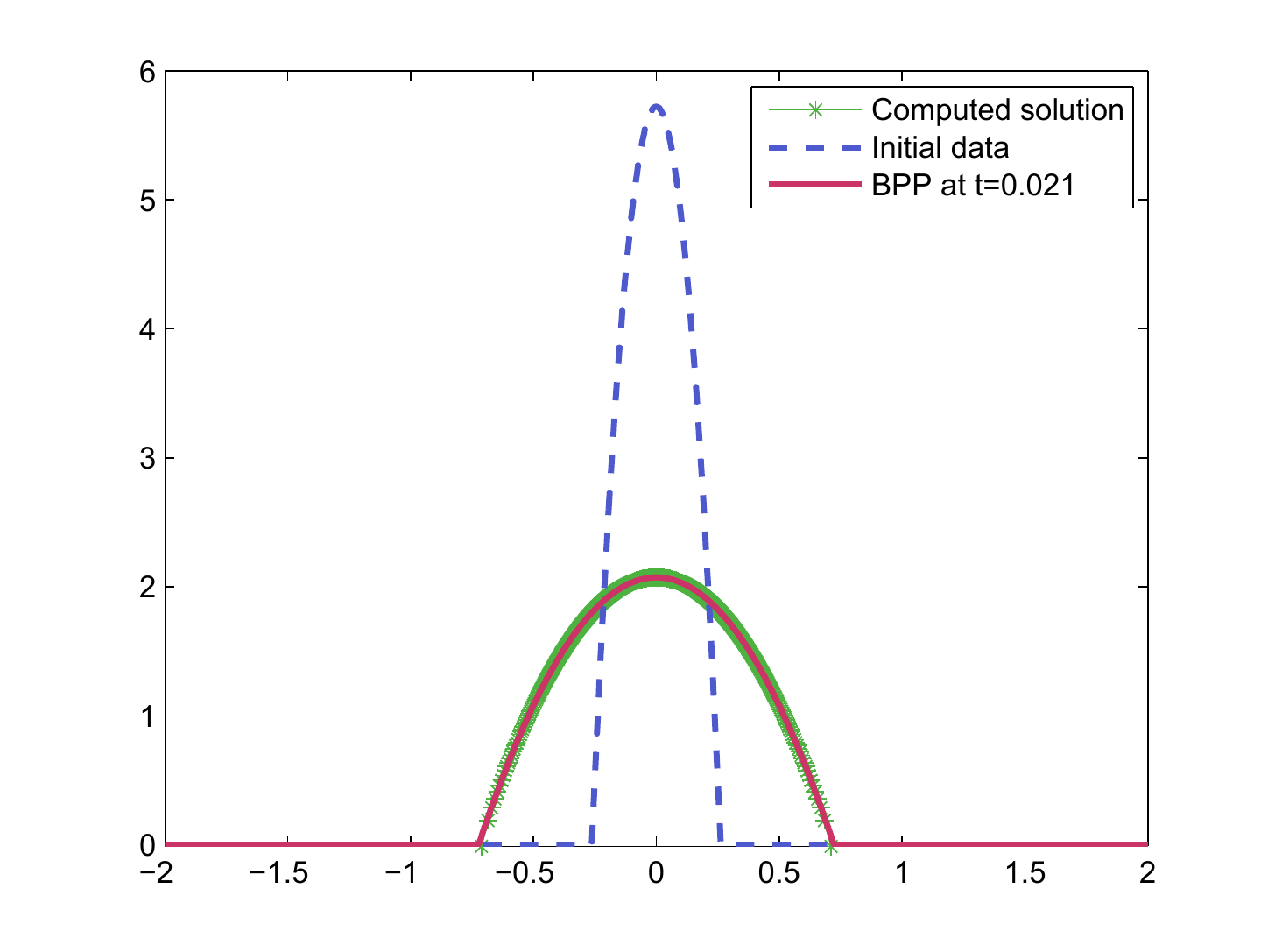}}\hspace*{0.5cm}
\subfigure[Evolution of the entropy in time with a logarithmic scale for the y-axis.]{\label{f:pme1_entropy}\includegraphics[width=0.42\textwidth]{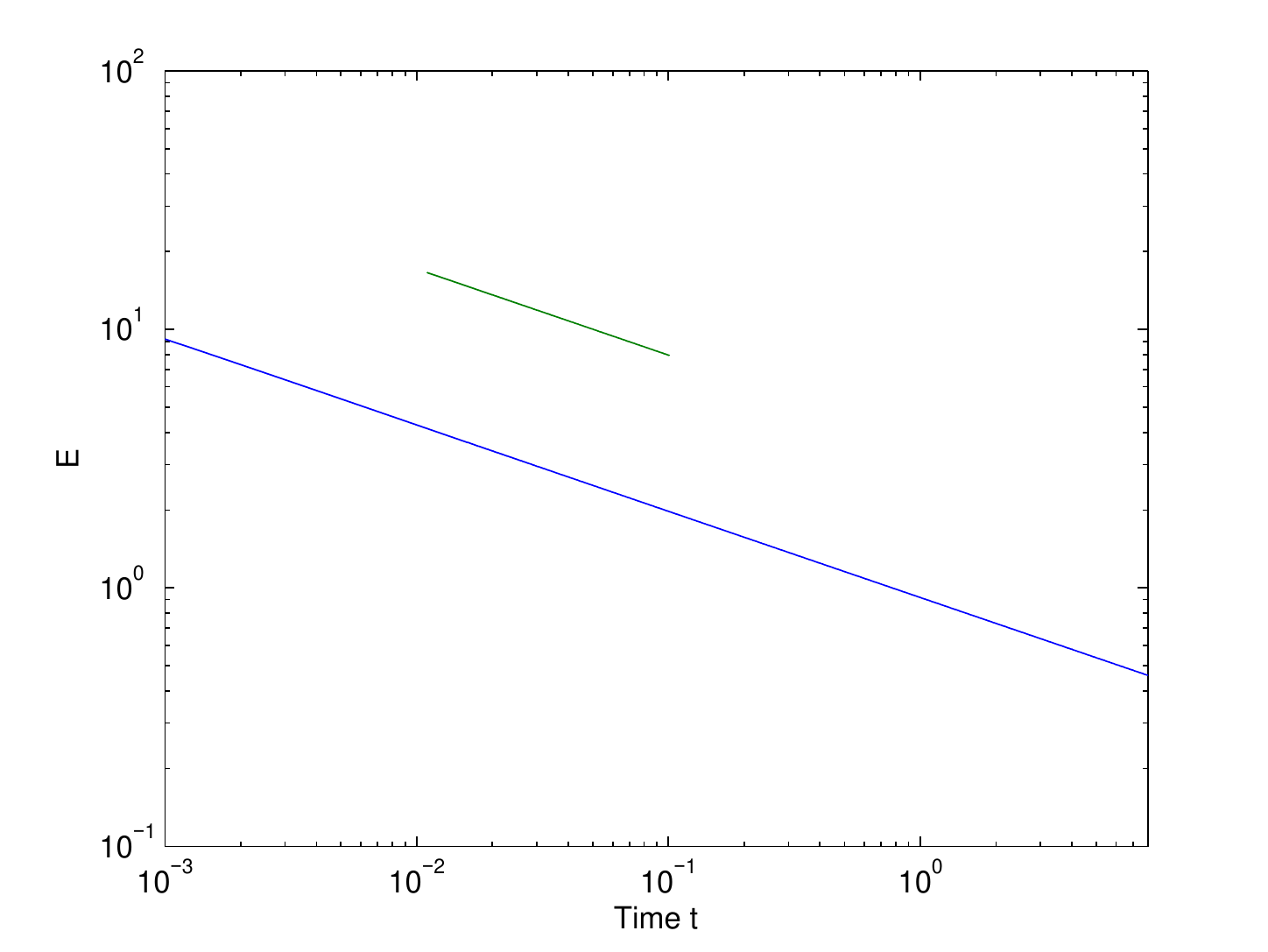}}
\caption{Solution of the PME for $m=2$ after 2000 time steps.} \label{f:pme1}
\end{center}
\end{figure}
In our first example we illustrate the behavior of our scheme for the porous medium equation (PME), which corresponds to $\Phi(s) = s^m, m > 1$ and $V=W=0$. 
The PME has self-similar solutions, the so-called Barenblatt Pattle profiles (BPP). We start our simulation with a BPP at time $t_0=10^{-3}$, i.e.
\begin{align*}
\rho_0(x) = \frac{1}{t_0^\alpha}\left(c-\alpha\frac{m-1}{2m}\frac{x^2}{t_0^{2\alpha}}\right)_+^{\frac{1}{m-1}},
\end{align*}
where $\alpha=\frac{1}{m+1}$ and $c$ is chosen such that $\int_{-1}^1 \rho_0(x)\,dx=2$.
Figures \ref{f:pme1} and \ref{f:pme2} show the density profiles $\rho$ at time $t=0.021$  for $m=2$ and $m=4$. It also illustrates the evolution of the free energy in time, 
which decays like $t^{-\alpha(m-1)}$. 
As seen from Figures \ref{f:pme1_entropy} and \ref{f:pme2_entropy}, these decays (indicated by the green lines) are perfectly captured by the scheme validating 
the chosen discretization. 
\begin{figure}[h!]
\begin{center}
  \subfigure[Reconstructed BP profile $\rho= \rho(x,t)$ at time $t = 0.021$.]{\includegraphics[width=0.42\textwidth]{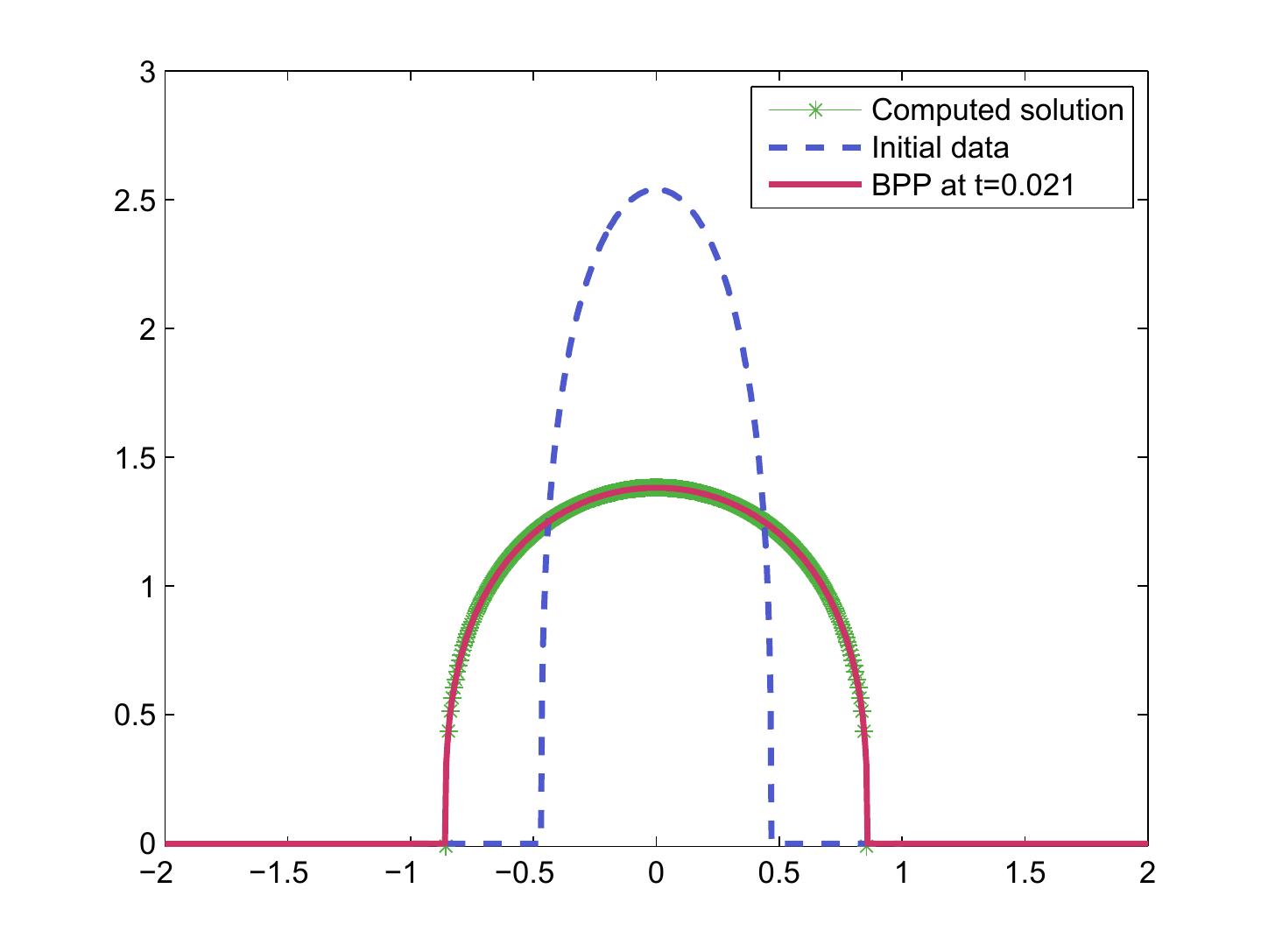}}\hspace*{0.5cm}
\subfigure[Evolution of the entropy in time with a logarithmic scale for the y-axis.]{\label{f:pme2_entropy}\includegraphics[width=0.42\textwidth]{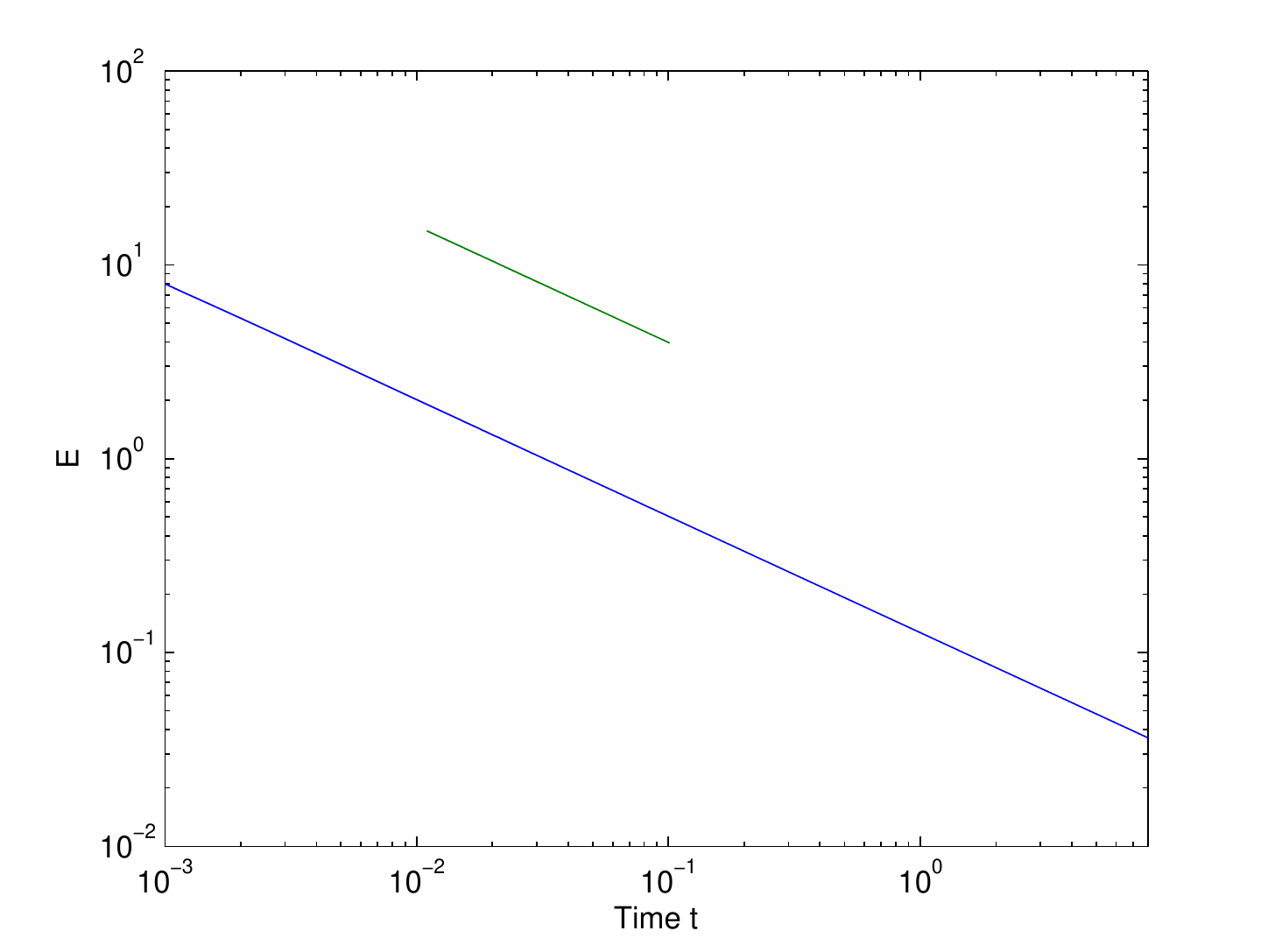}}
\caption{Solution of the PME for $m=4$ after 2000 time steps.} \label{f:pme2}
\end{center}
\end{figure}

Next we consider an agreggation equation with the logarithmic repulsive potential and harmonic confinement on the unit circle, that is $\Phi=0$ and
\begin{align*}
 W(x) = \frac{\lvert x \rvert^2}{2} - \ln \lvert x \rvert
\end{align*}
with an additional external potential of the form $V(x) = -\frac{\alpha}{\beta} \ln(\lvert x \rvert)$. 
\begin{figure}[h!]
\begin{center}
\subfigure[Transformed mesh]{\includegraphics[width=0.32\textwidth]{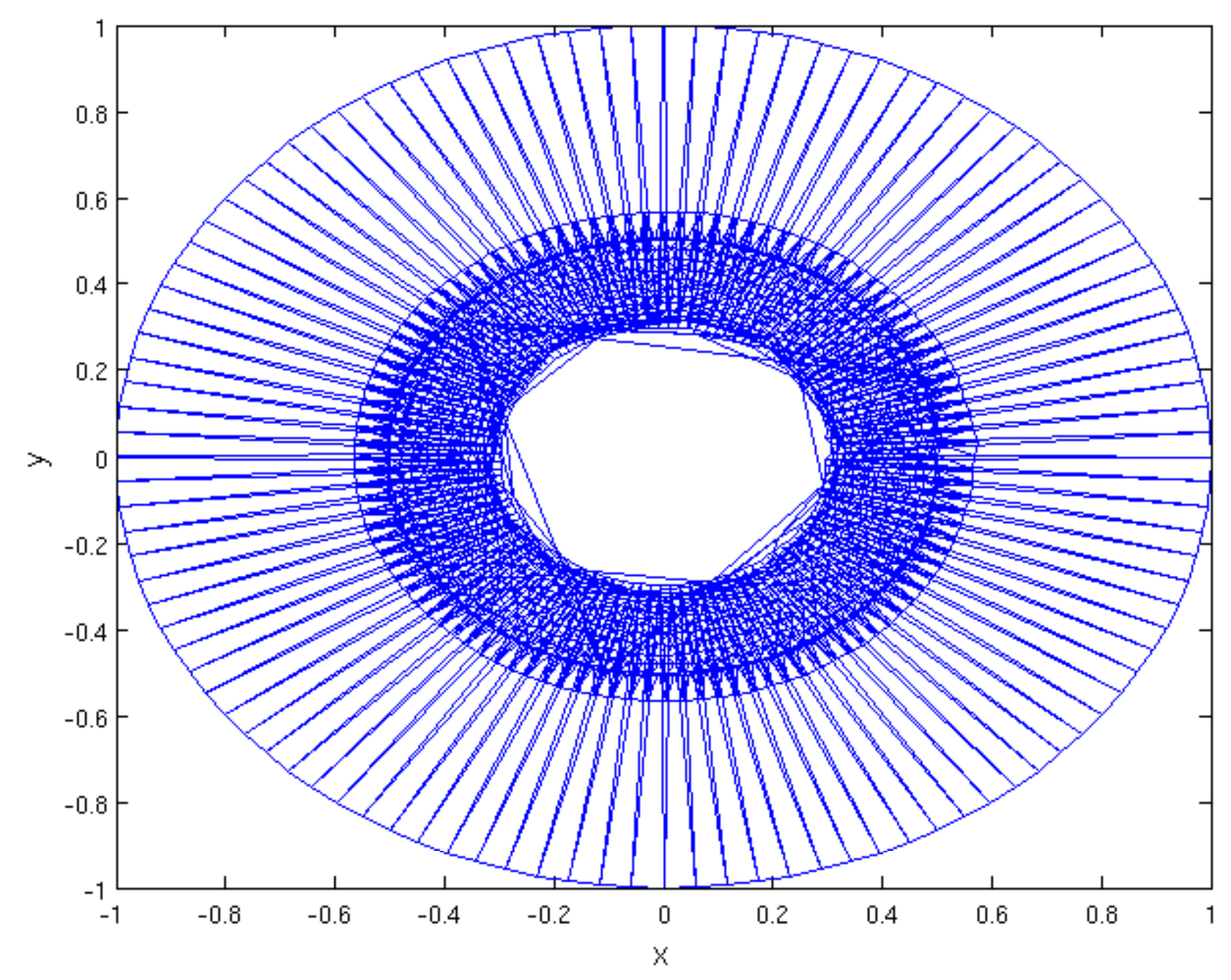}}
\subfigure[Density $\rho$]{\includegraphics[width=0.32\textwidth]{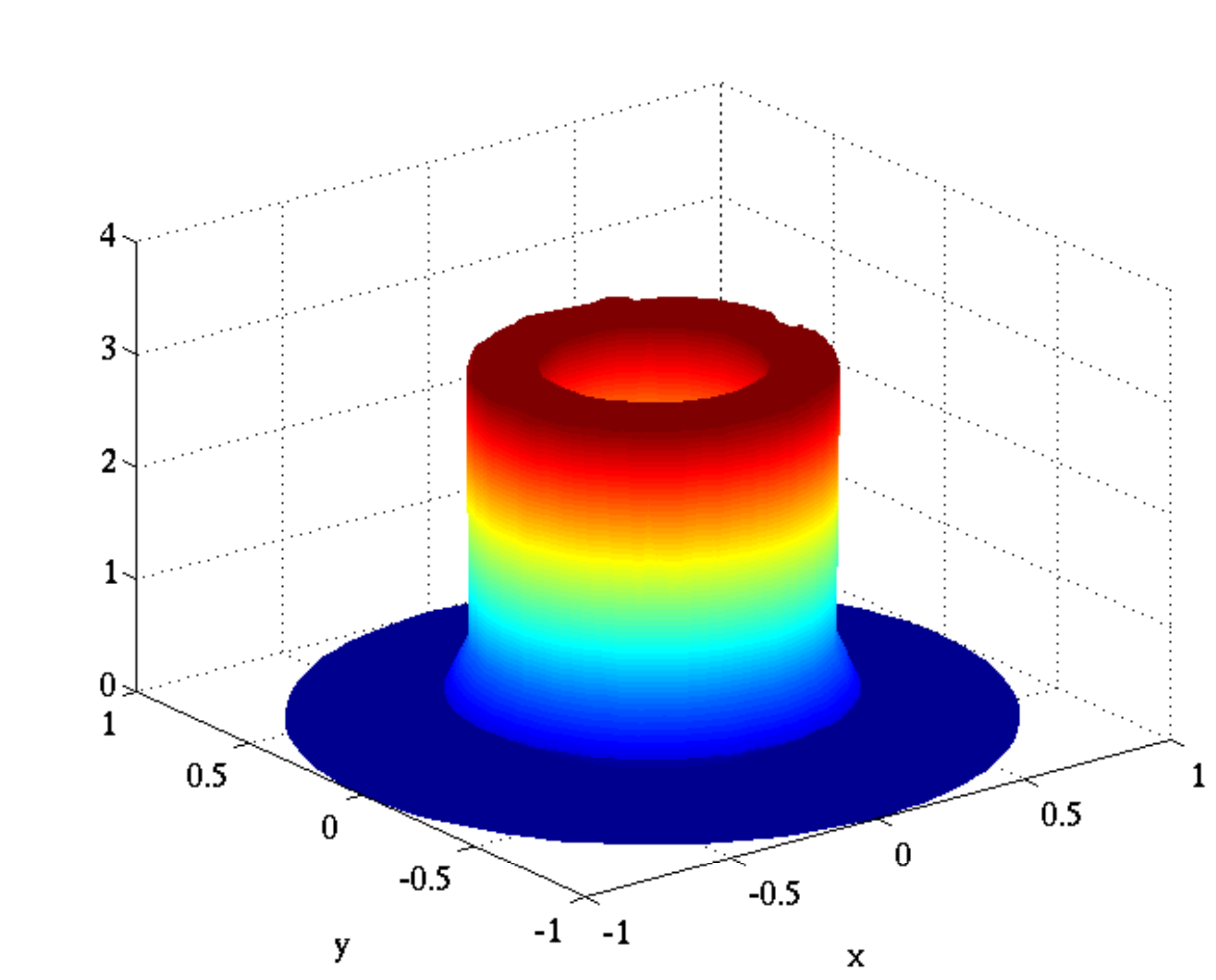}}
\subfigure[Relative entropy]{\includegraphics[width=0.32\textwidth]{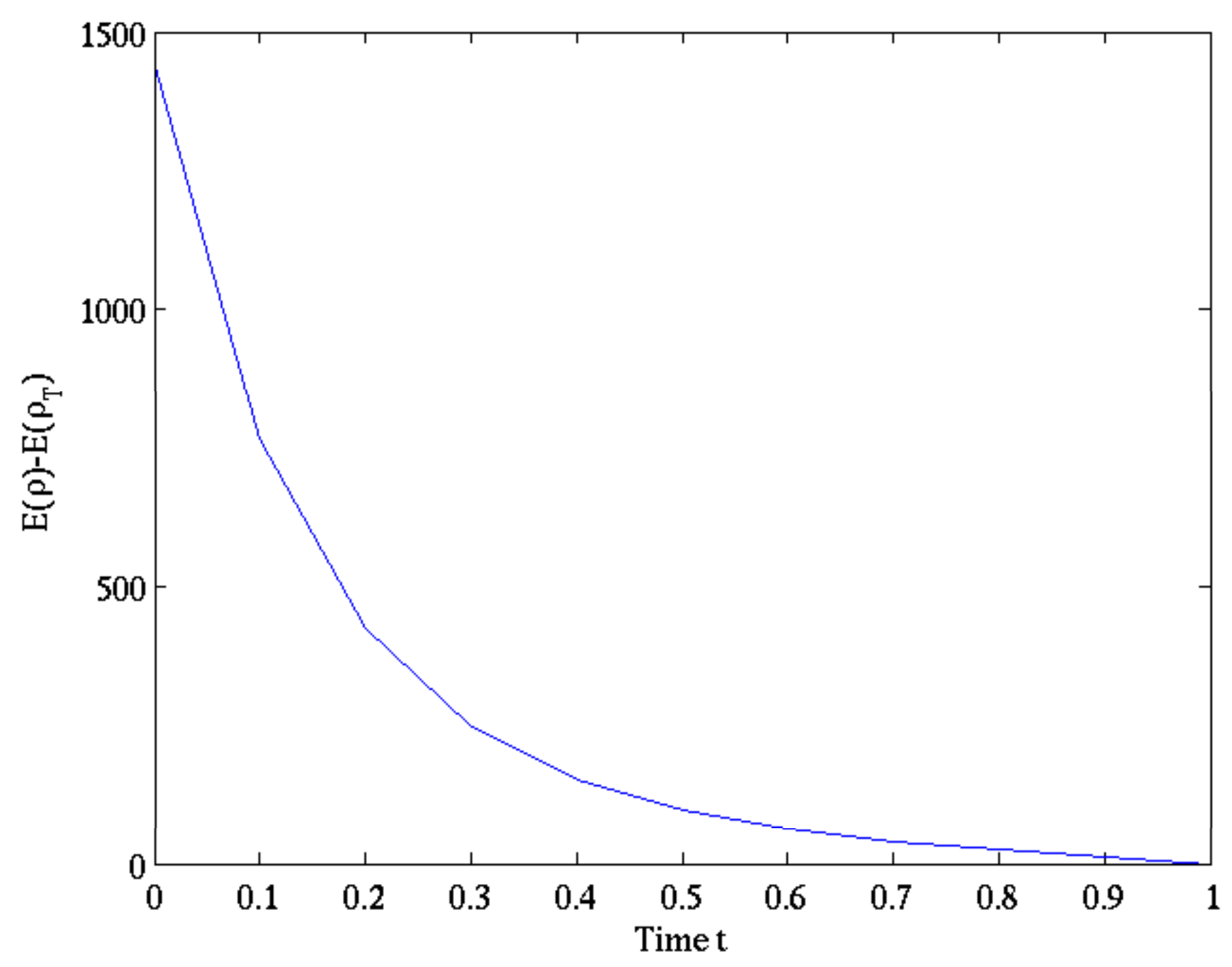}}
\caption{Simulation results for an attractive-repulsive potential $W = \frac{1}{2} \lvert x \rvert^2 - \ln(\lvert x \rvert)$ and an additional potential $V = \frac{1}{4} \ln (\lvert x \rvert)$.} \label{f:ex4}
\end{center}
\end{figure}
This model has been proposed as a way to find the spatial shape of the milling profiles in microscopic models for the dynamics of bird flocks.
Figure \ref{f:ex4} illustrates the corresponding annulus solution as well as the related computational challenges. The 'vacuum formation' at the center distorts the mesh and 
leads to degeneracies in the diffeomorphism.

In our final example we consider the KS model with an initial Gaussian of mass one 
and $\chi = 1.1 \times 8 \pi$.  We observe the expected blow up in Figure \ref{f:ex5}. Figure \ref{f:ex5}(c) indicates 
that the free energy decay is changing concavity as the free energy tries to decay faster possibly tending to $-\infty$ at the blow-up time.

\begin{figure}[h!]
\begin{center}
\subfigure[Transformed mesh]{\includegraphics[width=0.32\textwidth]{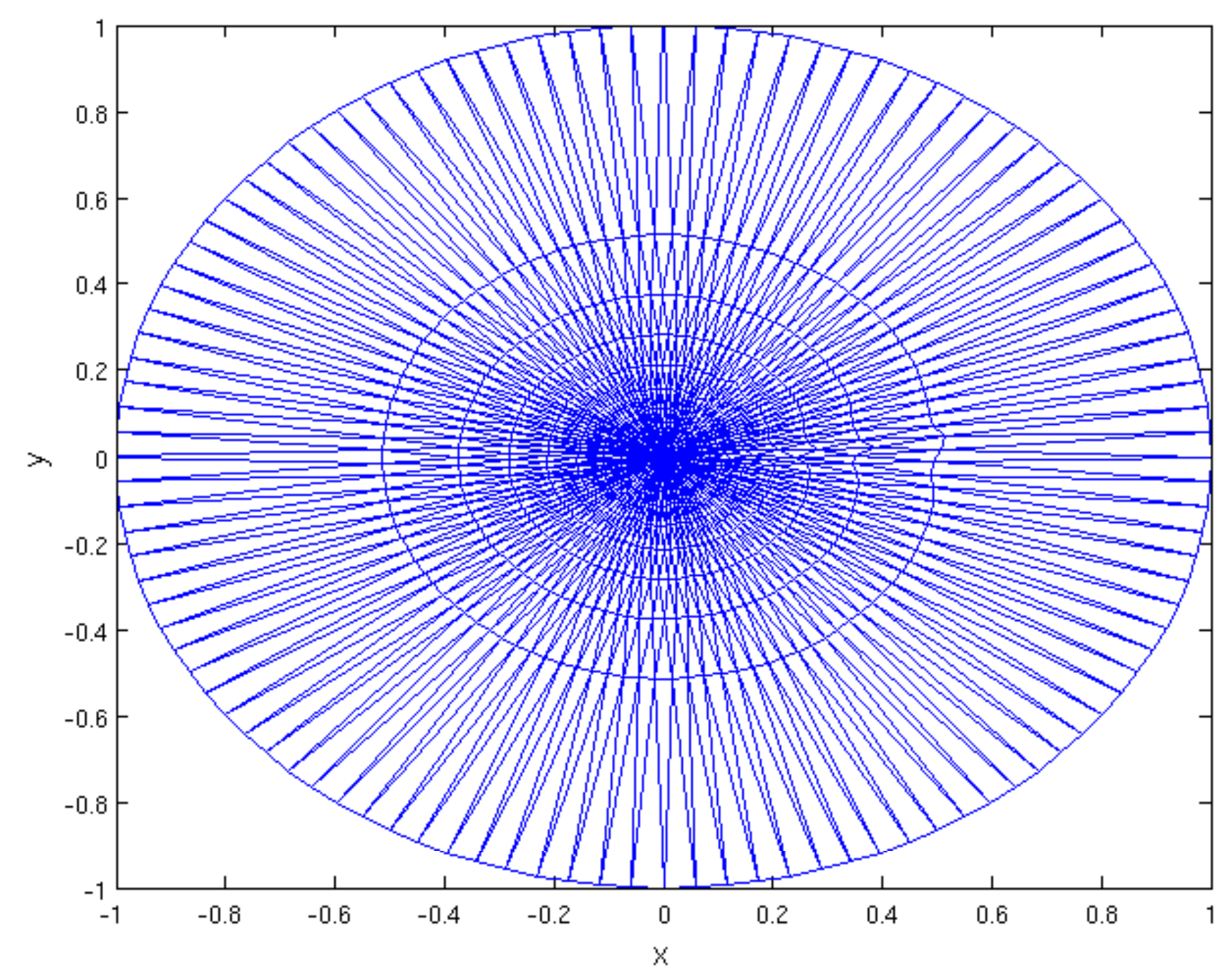}}
\subfigure[Density $\rho$]{\includegraphics[width=0.32\textwidth]{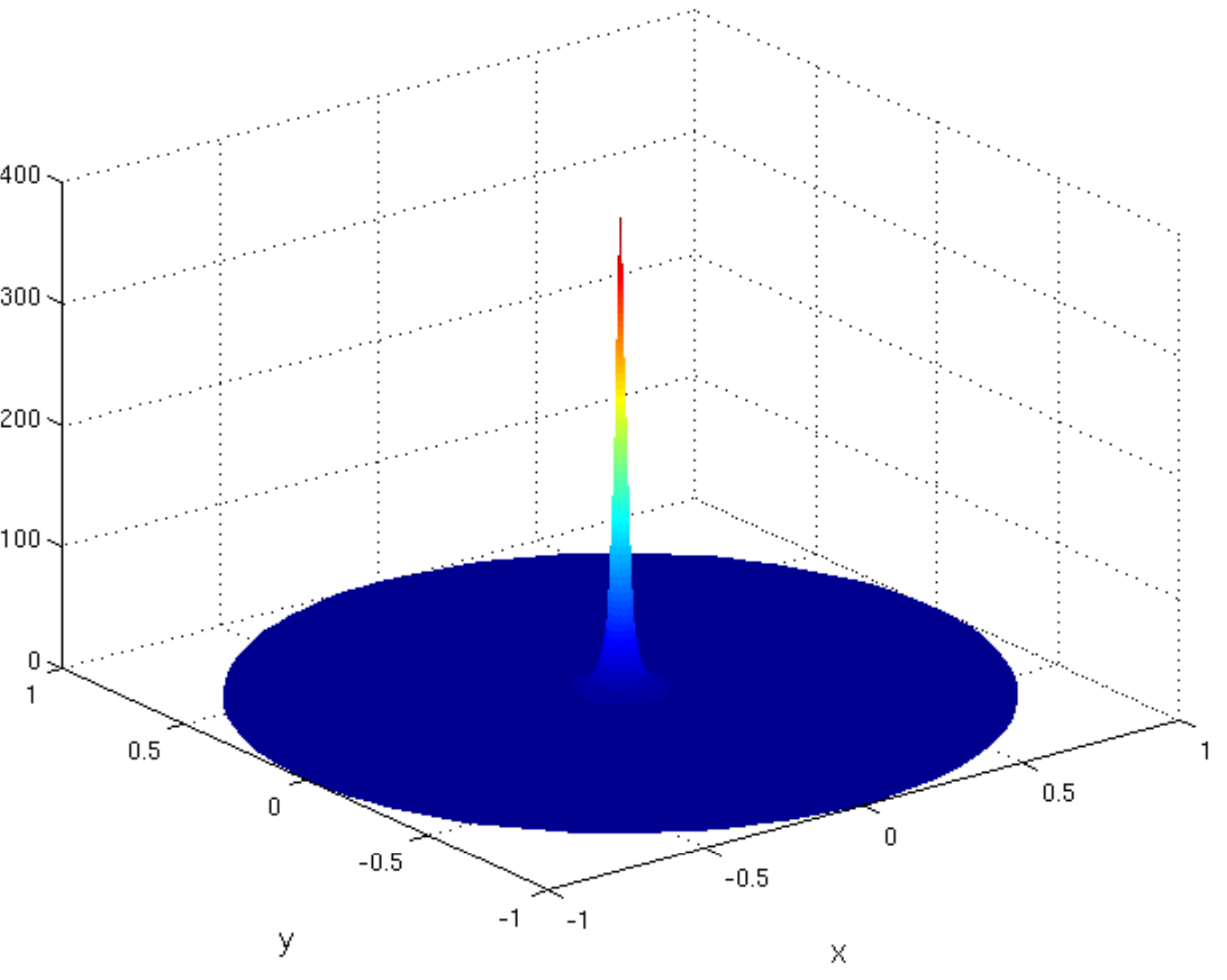}}
\subfigure[Relative entropy]{\includegraphics[width=0.32\textwidth]{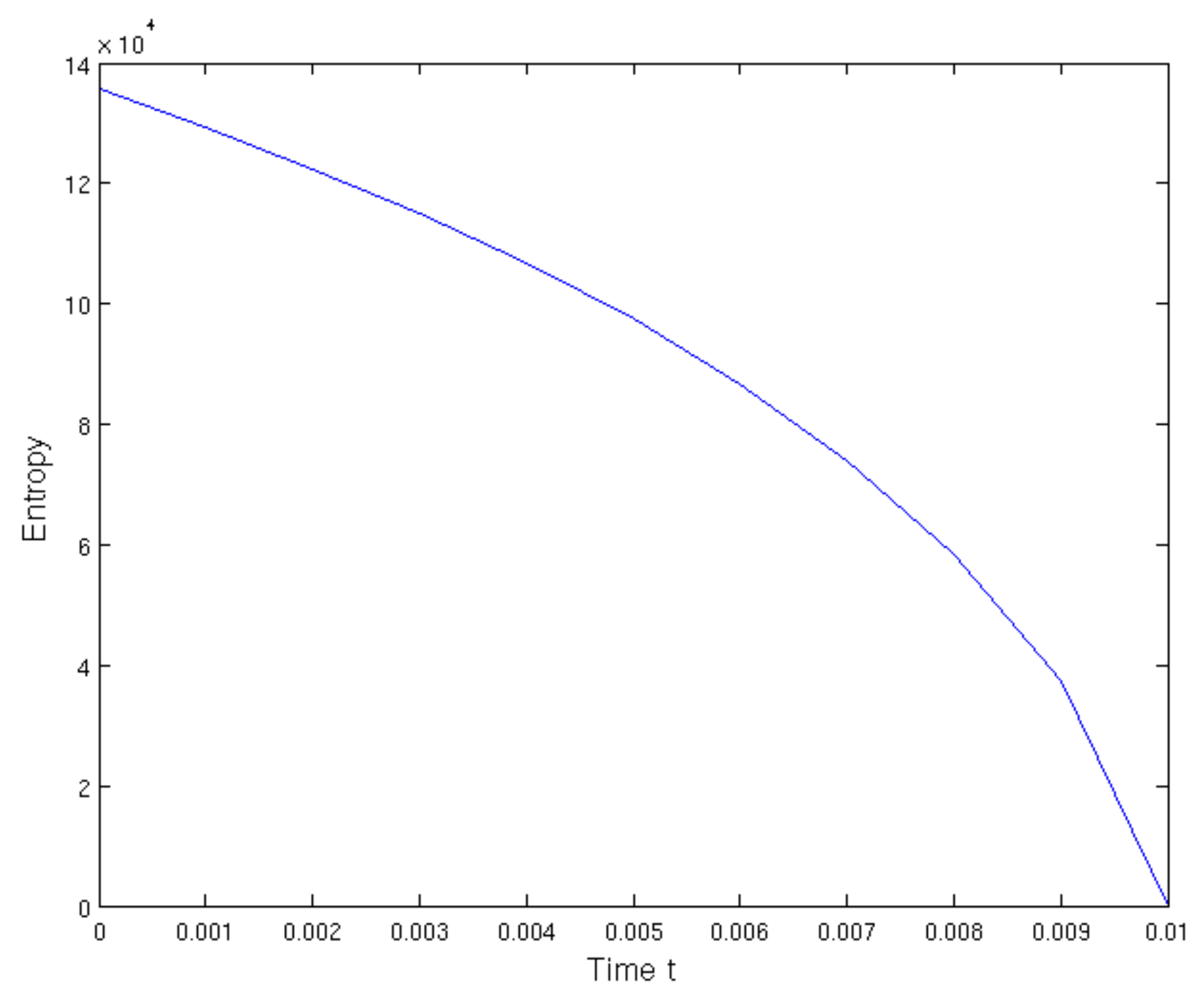}}
\caption{Simulation results for KS model with initial $M=1$ and $\chi = 1.1 \times 8\pi$.} \label{f:ex5}
\end{center}
\end{figure}


\subsection{1D: Finite Difference approach}

In this subsection, we briefly come back to the one dimensional problem as in the previous section and we show that certain finite differences method keep at the fully discrete level the asymptotic properties of the corresponding continuum problem. This general feature is quite desirable in general dimension for suitable approximations of the gradient flow equations \eqref{gf:0}. In fact, developing structure preserving schemes for general gradient flow equations is at the spotlight of research in this area, see \cite{bubba,bailo2018fully,CCP19} for finite volume schemes or the particle schemes discussed in subsection 1.4.4-5 above. To showcase these methods developed in \cite{gf-GT1,gf-GT2,gf-BCC,CM2009} for certain particular nonlinear diffusion equations, we use the modified Keller-Segel model in one dimension corresponding to linear diffusion $h(\rho)=\rho\log\rho$ with interaction potential $W(x)=\frac{\chi}\pi \log |x|$ in selfsimilar variables, that is, in the presence of a quadratic external potential $V(x)= \frac{x^2}2$ as in \cite{gf-BCC}.
The Euler-Lagrange optimality conditions \eqref{e:implicitdiffeo} can be rewritten in this particular case as
\begin{equation}\label{eq:eqfmoinsun}
  -\frac{X^{n+1}(w) - X^{n}(w)}{\Delta t} = \frac{\partial}{\partial w}
  \left[\left( \frac{\partial X^{n+1} (w)}{\partial w} \right)^{-1} \right]+
  X^{n+1}(w)+\frac{\chi}{\pi}  H[X^{n+1}]
\end{equation}
where $H$ corresponds to the Hilbert transform defined by
\begin{equation*}
 H[X](w) :=\frac1\pi \lim_{\eps \to 0} \int_{|X(w)-X(z)|\ge \eps} \frac{1}{X(w)-X(z)} \dd z\,.
\end{equation*}
For sake of simplicity we assume that we have an equidistant mass distribution of size $\Delta m$.
If we set $X^n_{i}:=X^n(i\Delta m)$, for any $i = 0\cdots N$, and $N\Delta m = 1$, the finite difference discretisation in space of~\eqref{eq:eqfmoinsun} is the following implicit Euler scheme 
\begin{multline} \label{eq:numerical scheme}
\displaystyle -\frac{X^{n+1}_{i} - X^{n}_{i}}{\Delta t} = \frac{1}{X^{n+1}_{i+1} - X^{n+1}_{i}} - \frac{1}{X^{n+1}_{i} - X^{n+1}_{i-1}} +X^{n+1}_{i}\\+ \frac{\chi}{\pi} \sum_{j\neq i|\ge \eps} \frac{\Delta m}{X^{n+1}_{i}-X^{n+1}_{j}}\;.
\end{multline}
We impose Neumann boundary conditions, {\it i.e.} for
any $n$, $\frac{1}{X^n_{N} - X^n_{N-1}} = 0$ and $\frac{1}{X^n_{1}
- X^n_{0}}=0$, so that the 'centre of mass' is conserved. That is for all $n$
$$ \displaystyle\sum_{i=0}^N X^n_i = 0.$$
The solution at each time step of the non-linear system of equations is
obtained by an iterative Newton-Raphson procedure as in the previous subsection.
The motivation for this numerical scheme is that we are able to show that for a fixed $\Delta t>0$, the
discrete solution converges to a unique steady state as time goes to infinity. To be precise the claim is the following: 
Assume $\chi<\chi_c$. Then the solution of the numerical scheme \eqref{eq:numerical
scheme} converges to the (unique) steady-state of the problem at an exponential rate.

First we need the following two characterizations of the (unique)
equilibrium state. The uniqueness will in fact follow from the
convergence proof, as we shall see later. The discrete function $(U_{i})$ is
an equilibrium if and only if for all $i$
\begin{equation} 0 = \frac{1}{U_{i+1} - U_{i}} - \frac{1}{U_{i}-U_{i-1}} + U_{i} + \frac\chi\pi \sum_{j\neq i}  \frac{\Delta m}{U_{i}-U_j},
\label{eq:discrete equilibrium 1} \end{equation}
or equivalently for all $k$
\begin{equation}
 \label{eq:discrete equilibrium 2}
(U_{k+1}-U_k) \left\{\frac\chi{\pi} \sum_{j=0}^k \sum_{i=k+1}^N  \frac{\Delta m}{U_{i}-U_j} - \sum_{i=0}^k U_{i}  \right  \} = 1.
\end{equation}
To see that~\eqref{eq:discrete equilibrium 1} and~\eqref{eq:discrete equilibrium 2} are equivalent, rewrite the latter as
\[ 
\forall k \quad  \frac1{U_{k+1}-U_k} = \frac\chi{\pi} \sum_{j=0}^k \sum_{i=k+1}^N  \frac{\Delta m}{U_{i}-U_j} - \sum_{i=0}^k U_{i}, 
\] 
and then 'derive' it in a discrete way. Let us assume the existence of this equilibrium for $\chi<\chi_c$, easy to obtain by a discrete minimization argument. We proceed as computing the discrete time evolution of the $L^2-$distance between $X^n$ and the stationary state $U$, that is, the discrete counterpart of the 2-Wasserstein distance between $X^n$ and $U$ given by
$$
\|X^{n} - U\|^2= \sum_i (X^{n+1}_{i} - U_i)^2.
$$
We deduce that
\begin{align*} 
\frac1{2\Delta t} \Big( \|X^{n+1} - U\|^2 \!- \|X^{n} - U\|^2 \Big)
\!& =
\!\frac1{2\Delta t} \sum_i (X^{n+1}_{i} - X^n_i) (X^{n+1}_{i}+X^n_i - 2U_{i}) \\
& \leq  \sum_i \frac{X^{n+1}_i - X^n_i}{\Delta t} (X^{n+1}_{i} - U_{i}) .
\end{align*} 
We then input the evolution equation for $X^{n+1}-X^n$, and obtain thanks to~\eqref{eq:discrete equilibrium 1},
\begin{align*} 
 \frac1{2\Delta t} \Big( \|X^{n+1} - U\|^2 - \|X^{n} - U\|^2 \Big) & \leq  -  \sum_{i} \Big( \frac{1}{X_{i+1} - X_{i}} - \frac{1}{X_{i}-X_{i-1}} -  \frac{1}{U_{i+1} - U_{i}} \\
& \quad \quad\quad+ \frac{1}{U_{i}-U_{i-1}}    + X_i - U_{i} + \frac\chi\pi\sum_{j\neq i} \frac{\Delta m}{X_i-X_j}\\
& \quad \quad\quad - \frac\chi\pi\sum_{j\neq i} \frac{\Delta m}{U_{i}-U_j}\Big) (X_i-U_{i}) \\
& =   A_n + B_n + C_n,
\end{align*} 
where $V$ stands for $X^{n+1}$ without any ambiguity.
We integrate by part the first (diffusion) contribution,
\begin{align*} 
 A &= - \sum_i \Big( \frac{1}{X_{i+1} - X_{i}} - \frac{1}{X_{i}-X_{i-1}} -  \frac{1}{U_{i+1} - U_{i}} + \frac{1}{U_{i}-U_{i-1}} \Big)(X_i-U_{i}) \\
& =   \sum_i \Big(\frac{1}{X_{i+1} - X_{i}} - \frac{1}{U_{i+1} - U_{i}}\Big)(X_{i+1} - U_{i+1} - X_i + U_{i}).
\end{align*} 
We have carefully used the boundary conditions. We can rewrite $A$ using zero-homogeinity of the last expression, namely
\[ 
A =  \sum_i \gamma\Big( \frac{X_{i+1} - X_{i}}{U_{i+1} - U_{i}} \Big), 
\]
where $\gamma (\lambda) = 2- \lambda - \lambda^{-1}$ is concave and non-positive.
The second contribution coming from variables rescaling is obvious but crucial, namely
$$
 B = -  \sum_i (X_i - U_{i})^2 = -  \|X^{n+1} - U\|^2.
$$
The last (interaction) contribution is given by
\begin{align*} 
 C & = - \frac\chi\pi\sum_i \Big( \sum_{j\neq i} \frac{\Delta m}{X_i-X_j} - \sum_{j\neq i} h \frac{\Delta m}{U_{i}-U_j} \Big) (X_i - U_{i}) \\
& = - \frac\chi{2\pi} \sum\sum_{i,j,\, i\neq j} {\Delta m} \Big(  \frac1{X_i-X_j} -  \frac1{U_{i}-U_j} \Big) (X_i-X_j - U_{i}+U_j) \\
& = - \frac\chi{2\pi} \sum\sum_{i,j,\, i\neq j} {\Delta m}\, \gamma \Big(\frac{X_i-X_j}{U_{i}-U_j}\Big).
\end{align*} 
We refer to \cite{gf-BCC} to check that the concavity of $\gamma$ shows that
\begin{align*} 
 C\leq  -  \sum_k \gamma\Big(\frac{V^{k+1}-V^k}{U^{k+1}-U^k}\Big) (U^{k+1}-U^k) \left\{\frac\chi{\pi} \sum_{j=0}^k \sum_{i=k+1}^N   \frac{\Delta m}{U_{i}-U_j} - \sum_{i=0}^k U_{i} \right\}.
\end{align*} 
The alternative representation of the stationary solution~\eqref{eq:discrete equilibrium 2} implies that $A+C \leq 0$. As a consequence we obtain
\begin{equation}
\label{eq:convergence estimation}
 \frac1{2\Delta t} \Big( \|X^{n+1} - U\|^2 - \|X^{n} - U\|^2 \Big)  \leq  -  \|X^{n+1} - U\|^2  .
\end{equation}
We finally get the exponential convergence rate,
\[ \|X^{n} - U\|^2 \leq \Big(\frac1{1+2\Delta t}\Big)^n\|X^0-U\|^2.\]
If $\Delta t$ is small, we can thus approximate $\log(1+2\Delta t)\approx
2\Delta t$ and $\Big(\frac1{1+2\Delta t}\Big)^n \approx
\exp(-2n\Delta t)\approx \exp(-2t)$. Thus, the bound on the rate, we
find, does not depend on the parameter $\chi<\chi_c$.

We can deduce {\em a posteriori} the
uniqueness of the equilibrium. As a matter of fact let consider
another equilibrium state $\tilde U$ and set $X^{n+1} = X^n =
\tilde U$ in the above computations. We eventually obtain
$\|\tilde U - U\| \leq 0$ from~\eqref{eq:convergence estimation},
which proves the uniqueness.


\section{Other variational approaches}\label{s:va}

In this final section, we want to briefly mention another recent line of research for computing the solutions of the variational scheme \eqref{gf:mm} and even calculate the optimal transportation problem, i.e., computing the geodesic curve joining two given densities. This line of research deals with the following  dynamic reformulation of the Wasserstein distance $\wass(\rho_0,\rho_1)$ due to Benamou and Brenier \cite{BB00}, where the distance is obtained as
\begin{align} \label{BB1}
&\wass(\rho_0,\rho_1)   = \inf_{(\rho,\velo) \in \C_0}  \left\{ \int_0^1 \int_\Omega |\velo(x,t)|^2 \,\rd \rho(x,t)  \rd t \right\}^{1/2} ,
\end{align}
where $(\rho,\velo)\in AC(0,1;\P(\Omega) )\times  L^1(0,1; L^2(\rho))$ belongs to the constraint set $\C_0$ provided that
\begin{align} \label{BB2}
 \partial_t \rho + \nabla \cdot (\rho \velo) = 0 &\,\,\text{ on } \Omega \times[0,1] \\
  (\rho \velo) \cdot \nml = 0   &\,\,\text{ on } \partial \Omega \times [0,1] ,\label{BB3p5} \\
 \rho(\cdot,0) = \rho_0, \ \rho(\cdot, \Dtime) = \rho_{\Dtime} &\,\,\text{ on } \Omega. \label{BB3}
\end{align}
A curve $\rho$ in $\P(\Omega)$ is \emph{absolutely continuous} in time, denoted $\rho \in AC(0,1;\P(\Omega) )$,  if there exists $w \in L^1(0,1)$ so 
that $\wass(\rho(\cdot, t_0), \rho(\cdot, t_1)) \leq \int_{t_0}^{t_1} w(s) \rd s$ for all $0 < t_0 \leq t_1 < 1$. The PDE constraint 
(\ref{BB2}-\ref{BB3p5}) holds in the duality with smooth test functions on $\Rd \times [0,1]$, i.e. for all $f \in C^\infty_c(\Rd \times [0,1])$,
\begin{align*}
\int_0^1  \int_\Omega \left[ \partial_t f(x,t) \rho(x,t) \right.&+\left. \grad f(x,t) \cdot  \velo(x,t) \rho(x,t) \right] \rd x \rd t \\
&+ \int_\Omega \left[ f(x,0)\rho_0(x) - f(x,1)\rho_1(x)\right] \rd x = 0\,.
\end{align*}
This dynamic reformulation reduces the problem of finding the Wasserstein distance between any two measures to identifying  the curve  in $\P(\Omega)$ that 
connects them with minimal kinetic energy.  Since  (\ref{BB1}) is not strictly convex, and the PDE constraint (\ref{BB2}) is nonlinear, in Benamou and Brenier's original work, 
they restrict their attention to the case $\rho(\cdot, t) \in \P_{ac}(\Omega)$ and introduce the momentum variables $m = \velo \rho$,  in order to rewrite (\ref{BB1}) as
\begin{align}\label{BB4}
\wass(\rho_0,\rho_1)^2 =  \min_{(\rho, m) \in \C_1}  \int_0^1 \int_\Omega \Phi(\rho(x,t), m(x,t))\rd x \,\rd t,
\end{align}
where
\begin{align*} 
\Phi(\rho, m) = 
\begin{cases}
 \frac{\|m \|^2}{ \rho} & \text{ if } \rho > 0\\
 0 & \text { if } (\rho,m) = (0,0)\\
 \infty & \text{ otherwise}
\end{cases}
\end{align*}
and $(\rho,m) \in AC(0,1; \P_{ac}(\Omega)) \times L^1(0,1;L^2(\rho^{-1}))$ belong to the constraint set $\C_1$ provided that
\begin{align*} 
 \partial_t \rho + \nabla \cdot m = 0 &\,\,\text{ on } \Omega \times[0,1]
 \\ m \cdot \nml= 0   &\,\,\text{ on } \partial \Omega \times [0,1] . 
 \\ \rho(\cdot,0) = \rho_0, \ \rho(\cdot, \Dtime) = \rho_{\Dtime} &\,\,\text{ on } \Omega . 
\end{align*}
After this reformulation, the integral functional  
\begin{align} \label{integralJ}
(\rho,m) \mapsto \int_0^1 \int_\Omega \Phi(\rho,m)
\end{align}
 is strictly convex along linear interpolations and lower semicontinuous with respect to weak-* convergence \cite[Example 2.36]{ambrosio2000functions}, and the PDE constraint is linear.
As an immediate consequence, one can conclude that minimizers are unique.
Furthermore, for any $\rho_0, \rho_1 \in \P_{ac}(\Omega)$, a direct computation shows that the minimizer $(\bar \rho, \bar m)$ is given by  the Wasserstein geodesic from
$\rho_0$ to $\rho_1$, see \cite{Villani03a, AGS2005,santambrogio2015optimal} for further background on optimal transport. Furthermore, given any minimizer $(\bar \rho, \bar m)$ of 
(\ref{BB4}), one can recover the optimal transport map in terms of $(\bar \rho, \bar v)$. Building upon Benamou and Brenier's  dynamic reformulation of the Wasserstein distance, one can also consider a dynamic reformulation of the JKO scheme \eqref{gf:mm}.  In particular, substituting (\ref{BB4}) in \eqref{gf:mm} leads to the following dynamic JKO scheme: given $\Delta t >0$, $\energy$, and $\rho_0$, solve the constrained optimization problem,
\begin{align*} 
\inf_{(\rho,m) \in \C}  \int_0^1 \int_\Omega \Phi(\rho(x,t), m(x,t))\,\rd x \rd t + 2\Delta t \energy(\rho(\cdot,1)) ,
\end{align*}
where $(\rho,m) \in AC(0,1; \P_{ac}(\Omega)) \times L^1(0,1;L^2(\rho^{-1}))$ belong to the constraint set $\C$ provided that
\begin{align}
 \partial_t \rho + \nabla \cdot m = 0 &\text{ on } \Omega \times[0,1] ,   \, m \cdot \eta = 0   \text{ on } \partial \Omega \times [0,1] , \, \text{ and } \rho(\cdot,0) = \rho_0  \text{ on } \Omega . \label{BB8}
\end{align}
This Eulerian approach to optimal transport has also been used  for numerical purposes. In fact, the first numerical methods for the Wasserstein distance were done 
in this formulation in the seminal paper of Benamou and Brenier \cite{BB00} by using the augmented Lagrangian method ALG2 for convex optimization problems. More recently, 
modern proximal splitting methods have been used by Papadakis, Peyre, and Oudet in \cite{PPO14}. Adding an additional Fisher information term in this dynamic formulation 
(in analogy with entropic regularization) has also been explored in \cite{LYO17}. For a detailed survey of state of the art methods in computational optimal transport, 
we refer the reader to the recent book by P\'eyre and Cuturi in \cite{PC18}. We also refer to the companion Chapter in this volume of Quentin Merigot.

However, this strategy has only been recently used for computing the Wasserstein distance integrated with the JKO scheme (\ref{gf:mm}) in order to simulate  partial differential equations of the form (\ref{gf:0}), see for instance \cite{BCMO16,BCL16,carrillo2019primal,li2019fisher}. The approach in \cite{carrillo2019primal} is to discretize first the optimality conditions in \eqref{BB8} in order to solve numerically the corresponding optimization problem at the discrete level. The convexity of the problem at the discrete level is conserved as soon as the energy $\energy$ is convex, since the functional related to the distance $\wass$ is convex and the constraints are linear. Unlike \cite{BCL16}, which applied Benamou and Brenier's classical ALG2 discretization to numerically approximate solutions of this minimization problem, the authors in \cite{carrillo2019primal} solve the optimization problem using a modern primal dual three operator splitting scheme due to Yan \cite{Yan17}. The previous work by Papadakis, Peyre, and Oudet \cite{PPO14} applied a similar two operator splitting scheme to simulate the Wasserstein distance. However, there are a few key differences in these two approaches. First,  the implementation of the primal dual splitting scheme in  \cite{carrillo2019primal} is done in a manner that does not require matrix inversion of the finite difference operator, which reduces the computational cost. In fact, the authors are able to obtain the exact expression for the proximal operator, which allows their method to be truly positivity preserving and very fast computationally.

Another key difference between these methods lies in the treatment of the linear PDE constraint in the dynamic reformulation of the Wasserstein distance. While most of previous work has imposed the linear PDE constraint \emph{exactly}, via a finite difference approximation, the authors in \cite{carrillo2019primal} allow the linear PDE constraint to hold up to an error of order $\delta$, which can be tuned according to the spatial discretization $(\Delta x)$ and the inner temporal discretization $(\Delta t)$ to respect the order of accuracy of the finite difference approximation of the continuity equation and the error in the initial and boundary constraints. Numerically, this allows their method to converge in fewer iterations, without any reduction in accuracy. Theoretically, this allows them to prove convergence of minimizers of the fully discrete problem to minimizers of the JKO scheme (\ref{gf:mm}), since minimizers of the fully discrete problem always exist, which is not the case when the PDE constraint is enforced exactly.
Finally, let us mention that the authors in \cite{li2019fisher} have explored the Fisher regularization of the optimal transportation cost leading to very good results and fast computational algorithms however, not solving exactly the JKO steps and therefore not being exactly energy decreasing at the fully discrete level. In summary, the Eulerian approach to solving the JKO minimization problem \eqref{gf:mm} is also being successfully and efficiently applied for numerical purposes due to the smart use of recent advances in numerical methods for optimization of convex functionals.


\section*{Appendix}

{\bf Energy on Lagrangian maps.} We here compute the variational derivative of the functional $\nrg^\#(X)$.

\begin{lemma}
  \label{gf:lem.cov}
  With $\nrg^\#(X)=\nrg(X\#\theta)$, we have at each diffeomorphism $\bar X:\Theta\to\Omega$:
  \begin{align}
    \label{gf:cov}
    \frac1\theta\frac{\delta\nrg^\#}{\delta X}(\bar X) = \nabla\frac{\delta\nrg}{\delta\rho}(\bar X\#\theta)\circ\bar X.
  \end{align}
\end{lemma}
\begin{proof}
  Let $\velo\in C^\infty_c(\Omega)$ be a smooth vector field,
  and define perturbations $(X^s)_{s\in\setR}$ of $\bar X$ via the flow of $\velo$,
  that is
  \begin{align*}
    \partial_sX^s = \velo\circ X^s,\quad X^0 = \bar X.
  \end{align*}
  On the one hand,
  \begin{align*}
    \frac{\dn}{\dd s}\bigg|_{s=0}\nrg^\#(X^s)
    = \int_\Theta \frac{\delta\nrg^\#}{\delta X}(X^s)\cdot\partial_sX^s\dd\xi\bigg|_{s=0}
    = \int_\Theta  \frac{\delta\nrg^\#}{\delta X}(\bar X)\cdot\velo\circ\bar X\dd\xi.
  \end{align*}
  And on the other hand, since
  \begin{align*}
    \partial_s(X^s\#\theta) = -\dv\big(X^s\#\theta\,\velo)
  \end{align*}
  by the properties of the push-forward,
  we have that
  \begin{align*}
    \frac{\dn}{\dd s}\bigg|_{s=0}\nrg(X^s\#\theta)
    &= \int_\Omega \frac{\delta\nrg}{\delta\rho}(X^s\#\theta)\,\partial_s(X^s\#\theta)\dd x\bigg|_{s=0}\\
    &= \int_\Omega \nabla\frac{\delta\nrg}{\delta\rho}(\bar X\#\theta)\cdot(\bar X\#\theta\,\velo)\dd x\\
    &= \int_\Theta  \left[\nabla\frac{\delta\nrg}{\delta\rho}(\bar X\#\theta)\right]\circ\bar X\cdot\velo\circ\bar X\,\theta\dd\xi.
  \end{align*}
  Since $\velo$ is arbitrary, and $\bar X$ is a diffeomorphism,
  the fact that $\nrg^\#(X^s)$ and $\nrg(X^s\#\theta)$ --- and hence also their $s$-derivatives --- are equal
  implies \eqref{gf:cov}.
\end{proof}

\bibliographystyle{alpha}
\bibliography{gv_numerics}

\end{document}